\documentclass[a4paper,reqno,12pt]{amsart}
\usepackage[left=2.75cm,right=2.75cm,top=2.85cm,bottom=2.85cm]{geometry}
\usepackage{amsmath,amssymb,latexsym,esint,cite,mathrsfs}
\usepackage{amsmath}
\usepackage{amssymb}
\usepackage{latexsym}
\usepackage{amsthm}
\usepackage{hyperref}
\usepackage{color}
\newcommand*\diff{\mathop{}\!\mathrm{d}}
\hypersetup{linkcolor=blue, colorlinks=true ,citecolor = red}
\begin{document}
	\title[ \MakeLowercase{p}-Laplacian in Exterior Domains]
	{The Fredholm Alternative for The \MakeLowercase{$p$}-Laplacian in exterior domains}
	
	\author[P. Dr\'abek, K. Ho \& A. Sarkar]
	{Pavel Dr\'{a}bek, Ky Ho and Abhishek Sarkar}
	\address{Pavel Dr\'{a}bek \newline
		Department of Mathematics, University of West Bohemia, Univerzitn\'i 8, 306 14 Plze\v{n}, Czech Republic}
	\email{pdrabek@kma.zcu.cz}
	\address{Ky Ho \newline
		NTIS, University of West Bohemia, Technick\'a 8, 306 14 Plze\v{n}, Czech Republic}
	\email{ngockyh@ntis.zcu.cz}
	\address{Abhishek Sarkar \newline
		NTIS, University of West Bohemia, Technick\'a 8, 306 14 Plze\v{n}, Czech Republic}
	\email{sarkara@ntis.zcu.cz}
	
	\subjclass[2010]{35J92, 35J60, 35J20, 35P30, 35J62, 35B40}
	\keywords{$p$-Laplacian; Fredholm alternative; the first eigenvalue; exterior domain; variational method}
	
	\begin{abstract}
		We investigate the Fredholm alternative for the $p$-Laplacian in an exterior domain which is the complement of the closed unit ball in $\mathbb{R}^N$ ($N\geq 2$). By employing techniques of Calculus of Variations we obtain the multiplicity of solutions. The striking difference between our case and the entire space case is also discussed.
	\end{abstract}
	
	\maketitle
	\numberwithin{equation}{section}
	\newtheorem{theorem}{Theorem}[section]
	\newtheorem{lemma}[theorem]{Lemma}
	\newtheorem{proposition}[theorem]{Proposition}
	\newtheorem{corollary}[theorem]{Corollary}
	\newtheorem{definition}[theorem]{Definition}
	\newtheorem{example}[theorem]{Example}
	\newtheorem{remark}[theorem]{Remark}
	\allowdisplaybreaks
	
	\section{Introduction}
	
	The Fredholm alternative for the $p$-Laplacian has been studied on both bounded domains in $\mathbb{R}^N$ and the entire space $\mathbb{R}^N.$ In this paper we investigate the existence and multiplicity of solutions of the following problem
	\begin{align}\label{1.1}
	\begin{cases}
	-\Delta_pu=\lambda K(x)|u|^{p-2}u+h \quad \text{in } B_1^c ,\\
	u=0\quad \text{on } \partial B_1 ,
	\end{cases}
	\end{align}
	where $\Delta_pu:=\operatorname{div}\left(  |\nabla u| ^{p-2}\nabla u\right)$ is the $p$-Laplacian with $p>1$, $B_1^c$ is the complement of the closed unit ball $B_1$ in $\mathbb{R}^N,$ $\lambda>0$ is a parameter, the weight $K$ and the function $h$ will be specified later.
	
	In a bounded domain $\Omega$ of $\mathbb{R}^N,$ similar problems (with $K(x)\equiv 1$) have been studied in numerous papers. For the references we refer the reader to survey papers by Tak\'{a}\v{c}\cite{Takac2006,Takac2010} and the references therein.
	
	In the case of the entire space $\mathbb{R}^N,$ Alziary et al. \cite{Alziary} studied the solvability of the equation
	\begin{equation}\label{I.eq.entire}
	-\Delta_pu=\lambda m(x)|u|^{p-2}u+f \quad \text{in } \mathbb{R}^N,\ u\in \mathcal{D}^{1,p}(\mathbb{R}^N),
	\end{equation}
	where $1<p<N$ and the Sobolev space $\mathcal{D}^{1,p}(\mathbb{R}^N)$ is defined to be the completion of $C_c^1(\mathbb{R}^N)$ with respect to the norm
	$$\|u\|_{\mathcal{D}^{1,p}(\mathbb{R}^N)}=\left(\int_{\mathbb{R}^N}|\nabla u|^p \diff x\right)^{1/p}.$$
	They studied problem \eqref{I.eq.entire} with a radially symmetric and measurable weight $m(x)=m(|x|)$ satisfying
	\begin{equation}\label{I.weight.RN}
0<m(r)\leq \frac{C}{(1+r)^{p+\mu}}\quad \text{a.e. in}\  [0,\infty),
	\end{equation}
	with some constants $\mu>0$ and $C>0.$ 
	Let $\widetilde{\lambda}_1>0$ be the first eigenvalue and $\widetilde{\varphi}_1$ be the corresponding positive eigenfunction of $-\Delta_p$ in $\mathbb{R}^N$ relative to the weight $m(|x|)$; for the existence of the first eigenpair see, for example \cite{Stavrakakis,Lucia} and the references therein. For a given $f^{\ast} \in [\mathcal{D}^{1,p}(\mathbb{R}^N)]^*$ (the dual space of $\mathcal{D}^{1,p}(\mathbb{R}^N))$, satisfying $\langle f^{\ast}, \widetilde{\varphi}_1 \rangle=0$ (where $\langle \cdot,\cdot \rangle$ denotes the duality pairing between $\mathcal{D}^{1,p}(\mathbb{R}^N)$ and $[\mathcal{D}^{1,p}(\mathbb{R}^N)]^*$), the authors of \cite{Alziary} obtained the existence of at least one solution of \eqref{I.eq.entire} for $2\leq p<N$ with $\lambda=\widetilde{\lambda}_1$ and $f = f^{\ast}$, and for $1<p<2\leq N$ with $\lambda \in ( \widetilde{\lambda}_1 -\epsilon, \widetilde{\lambda}_1+\epsilon),$ $\epsilon >0$ small, and $f$ in a neighbourhood of $f^{\ast}$. 
	
	To obtain the existence of solutions, the authors of \cite{Alziary} used variational arguments but treated the two cases $1<p<2\leq N$ and $2\leq p<N$ in a different way. As a by-product, for the resonant case $\lambda =\widetilde{\lambda}_1$, they proved ``a saddle point geometry" of the energy functional associated with \eqref{I.eq.entire} when $1<p<2\leq N$. On the other hand, they used an improved Poincar\'e inequality when $p\leq 2<N$ and showed that the energy functional has a ``global minimizer geometry".
	
	In the case of an exterior domain, Anoop et al. \cite{Anoop.NA} discussed the existence of solution of problem \eqref{1.1} with a weaker assumption on weight than in \cite{Alziary} (see Definition~\ref{II.def.weight} in the next section). By using the Fredholm alternative for the $p$-Laplacian due to Fu\v{c}\'{i}k et al. \cite[Chapter II, Theorem 3.2]{Fucik}, they obtained the existence of solutions for problem \eqref{1.1} when $\lambda\in (0,\lambda_1+\delta)\setminus\{\lambda_1\}$ for some $\delta>0,$ where $\lambda_1$ is the first eigenvalue of $-\Delta_p$ in $B_1^c$ relative to the weight $K$ (see \cite[Proposition 3.1]{Anoop.NA}).
	
	The goal of this paper is to obtain 
	multiple solutions of \eqref{1.1} for the resonant case $\lambda=\lambda_1$ with a weaker assumption on the weight than in \cite{Alziary}. This work can be seen as a complement to the Fredholm alternative for the $p$-Laplacian in an exterior domain  for the resonant case. It is worth mentioning that to deal with the resonant case, we apply the second order Taylor formula for the energy functional associated with \eqref{1.1} at the first eigenfunction $\varphi_1$ of $-\Delta_p$ in $B_1^c.$ To apply Taylor formula, we need to employ weighted spaces in terms of $\varphi_1$ with the weights singular or degenerate, on the set $\{\nabla \varphi_1=0\}.$ Surprisingly, the case of an exterior domain differs substantially from the case of  the entire space $\mathbb{R}^N$. The important point to note here is the fact that, if $K$ is radially symmetric and satisfies certain decay condition, the set $\{\nabla \varphi_1=0\}$ is a removable set (i.e., the set of zero capacity) in the case of the entire space $\mathbb{R}^N$, whereas this is not true in the case of an exterior domain 
	(see, e.g., Remarks~\ref{II.rmk.WeakDerivative}, \ref{II.rmk.proof_of_simplicity} and \ref{III.rmk.compare}). For this reason, to obtain a saddle point geometry of the energy functional in the resonant case when $1<p<2$, we need to introduce a new condition for the source term $h$, which is of independent interest. 
	
	The rest of the paper is organized as follows. In Section~\ref{Sec.Preliminaries},
	we review properties of the first eigenpair of $-\Delta_p$ in $B_1^c$ obtained in \cite{Anoop.CV, Anoop.NA} and then we prove more properties of the first eigenfunction. In this section we also introduce suitable weighted function spaces. In Section~\ref{Poincare.Inequality}, by employing weighted spaces introduced in the previous section we obtain an improved Poincar\'e inequality (Proposition~\ref{II.Prop.Poincare}) for our solution space when $2<p<N$. 
	In Section~\ref{Saddle.point.property}, we establish a saddle point geometry of the energy functional (Proposition~\ref{III.Prop.saddle}) in the resonant case when $1<p<2$. Section~\ref{Sec.Existence} is devoted to the investigation of the existence and multiplicity of solutions for \eqref{1.1}. In this section we complete the Fredholm alternative for the $p$-Laplacian in an exterior domain. More precisely, when $\lambda=\lambda_1$ and the source term $h$ is in a neighbourhood of given $h^\ast$ satisfying $\langle h^\ast,\varphi_1\rangle=0$ we obtain a solution for problem \eqref{1.1} by using the saddle point geometry of the energy functional and the improved Poincar\'e inequality when $1<p<2$ and $2<p<N$, respectively. If in addition the source term $h$ satisfies $\langle h,\varphi_1\rangle\ne 0$, we obtain a second solution for problem \eqref{1.1} that is a Mountain Pass type solution. For $p=2$ we recover the classical Fredholm alternative for the Laplace equations in an exterior domain. It is worth mentioning that the conditions on the weight $K$ and the dimension $N$ are relaxed in the linear case. Our main results are stated in Theorems~\ref{V.Theo.singular} and \ref{V.Theo.degenerate}. Finally, we provide proofs of several auxiliary results in Appendices \ref{AppendixA}--\ref{AppendixD}. 
	
	
	\section{Abstract framework and preliminary results}\label{Sec.Preliminaries}
	\subsection{The solution space} We study problem \eqref{1.1} with an admissible weight $K$ defined as follows. 
	\begin{definition}\rm\label{II.def.weight}
		We say that $K$ is \emph{admissible} if $K\in L_{loc}^1(B_1^c)$, $\operatorname{meas}\{x\in B_1^c: K(x)>0\}>0$ and there exists a positive function $w$ such that 
		\begin{itemize}
			\item[(i)]  $w\in \begin{cases}
			L^1((1,\infty);r^{p-1}),\ p\ne N,\\
			L^1((1,\infty);[r\log r]^{N-1}),\ p=N;
			\end{cases}$
			\item[(ii)] $|K(x)|\leq w(|x|)$ for a.e. $x\in B_1^c.$ 
			
		\end{itemize}  
	\end{definition}
	We look for solutions of \eqref{1.1} in $\mathcal{D}_0^{1,p}(B_1^c),$ the completion of $C_c^1(B_1^c)$ ($C^1$ functions with compact support) with respect to the norm
	$$\|u\|=\left(\int_{B_1^c}|\nabla u|^p\diff x\right)^{1/p}.$$ This space is a well defined uniformly convex Banach space with the following properties. 
	\begin{lemma}[\cite{Anoop.CV}]
		The following embeddings hold:
		\begin{itemize}
			\item[(i)]$\mathcal{D}_0^{1,p}(B_1^c)\hookrightarrow L_{loc}^p(B_1^c),$
			\item[(ii)]$\mathcal{D}_0^{1,p}(B_1^c)\hookrightarrow W_{loc}^{1,p}(B_1^c),$
			\item[(iii)]$\mathcal{D}_0^{1,p}(B_1^c)\hookrightarrow \hookrightarrow L^{p}(B_1^c;w(|x|)).$
		\end{itemize}
	\end{lemma}
Throughout this paper, we denote $$X:=\left(\mathcal{D}_0^{1,p}(B_1^c),\|\cdot\|\right),$$
and by $X^\ast$ the dual space of $X$.	The following definition of (weak) solution makes sense, thanks to the embedding, $X\hookrightarrow L^{p}(B_1^c;w(|x|)).$
	\begin{definition}\rm\label{II.def.weaksolution}
	Let $K$ be an admissible weight and let $h\in X^\ast$. By a \emph{ (weak) solution} of problem \eqref{1.1}, we mean a function $u\in X$ satisfying
		$$\int_{B_1^c}|\nabla u|^{p-2}\nabla u\cdot \nabla v \diff x=\lambda \int_{B_1^c}K(x)|u|^{p-2}uv\diff x+\langle h,v\rangle, \quad \forall v\in  X,$$ where $\langle .,. \rangle$ denotes the duality pairing between $X$ and $X^{\ast}$.\\
		When $h\equiv 0,$ $\lambda$ is called an \emph{eigenvalue} of $-\Delta_p$ in $B_1^c$ related to the weight $K$ (an eigenvalue, for short) if problem \eqref{1.1} has a nontrivial solution $u,$ and such a solution $u$ is called an \emph{eigenfunction} corresponding to the eigenvalue $\lambda.$
	\end{definition}
	In what follows, for $1<\alpha<\beta$ set
	\begin{align*}
	B_\alpha&:=\{x\in\mathbb{R}^N: |x|\leq \alpha\}, \ B_{\alpha}^c:= \mathbb{R}^N \setminus B_{\alpha}, \\ A_\alpha^\beta:=\{x\in\mathbb{R}^N&: \alpha< |x|<\beta \},\ S_\alpha:=\{x\in\mathbb{R}^N: |x|=\alpha\},
	\end{align*}
	and by $|S|$ denote the Lebesgue measure of $S\subset \mathbb{R}^N.$ For a normed linear space $E$, the symbol 
	$B_E(u,\rho)$ stands for 
	 the open ball centered at $u$ with radius $\rho$ in $E$.
	
	\subsection{Properties of the first eigenpair $(\lambda_1,\varphi_1)$}
	
	It was shown in \cite{Anoop.CV, Anoop.NA} that, for an admissible weight $K$ 
	we have 
	$$\lambda_1:=\inf \left\{\int_{B_1^c}|\nabla u|^p\diff x:\ u\in X,\int_{B_1^c}K(x)|u|^p\diff x=1\right\}>0.$$
	It is a simple eigenvalue of
	\begin{equation}\label{II.eigen}
	\begin{cases}
	-\Delta_pu=\lambda K(x)|u|^{p-2}u \quad \text{in } B_1^c ,\\
	u=0\quad \text{on } \partial B_1.
	\end{cases}
	\end{equation}
	Furthermore, the infimum above is achieved at an eigenfunction $\varphi_1$, which is positive  a.e. in $B_1^c.$ If we assume, in addition, $K\in L^s(A_1^R)\cap L_{loc}^\infty(B_1^c)$ for some $s>\frac{N}{p}$ and $R>1$ when $1<p\leq N$ or $K \in L_{loc}^\infty(B_1^c)$ when $p>N$ then  $\lambda_1$ is an isolated eigenvalue and $\varphi_1\in C^1(B_1^c)$ and $\varphi_1>0$ in $B_1^c$. If  $K\in L^\infty(A_1^R)$ for all $R>1$ then $\varphi_1\in C^{1,\alpha}(\overline{A_1^R})$ for all $R>1,$ where $\alpha=\alpha(R)\in (0,1).$ 
	Thus, applying the strong maximum principle by V\'azquez \cite[Theorem 5]{Vazquez} to
	$$-\Delta_p \varphi_1+\lambda_1 \|K\|_{L^\infty(A_1^R)}\varphi_1^{p-1}\geq -\Delta_p \varphi_1+\lambda_1 K^-(x)\varphi_1^{p-1}=\lambda_1 K^+(x)\varphi_1^{p-1}\geq 0\quad \text{in}\ A_1^R,$$
	where $K^+=\max\{K,0\}, K^-=K^+-K$, we have
	\begin{equation*}\label{II.reg}
	\frac{\partial \varphi_1}{\partial \nu}
	(x)<0,\ \forall x\in \partial B_1,
	\end{equation*}
	where $\nu$ is the unit outward normal vector to $\partial B_1$ at $x.$ From these facts, if  $K\in L^\infty(A_1^R)$ for all $R>1$, we deduce that 
	$$\mathcal{A}:=\{x\in B_1^c: \nabla\varphi_1(x)=0\},$$
	is a closed set in $\mathbb{R}^N$ and 
	$$\operatorname{dist}(\mathcal{A},\partial B_1)>0.$$
	Clearly, if the admissible weight $K$ is positive a.e. in $B_1^c$ then $\operatorname{int}(\mathcal{A})=\emptyset.$
	Moreover, $|\mathcal{A}|=0$ if we assume a stronger assumption on $K$ as shown in the following lemma.
	
	\begin{lemma}\label{II.meas.}
		Assume that the weight $K$ satisfies
		\begin{itemize}
			\item[(A)] $K$ is an admissible weight such that $K>0$ a.e. in $B_1^c$, $K\in L^s(A_1^R)\cap L_{loc}^\infty(B_1^c)$ for some $s>\frac{N}{p}$ and $R>1$ when $1<p\leq N$ or $K\in L_{loc}^\infty(B_1^c)$ when $p>N$.
		\end{itemize} Then 
		$$|\mathcal{A}|=0.$$
	\end{lemma}
	
	\begin{proof}
		Note that, $\varphi_1\in C^1(B_1^c)$ and $\varphi_1(x)>0$ for all  $x\in B_1^c.$ Set $f:=\lambda_1 K \varphi_1^{p-1}.$ Then for each $n\in\mathbb{N}\setminus\{1\},$ $\varphi_1$ satisfies
		$$\int_{A_{1+1/n}^n}|\nabla\varphi_1|^{p-2}\nabla\varphi_1\cdot \nabla v \diff x= \int_{A_{1+1/n}^n}fv\diff x, \quad \forall v\in  C_c^\infty(A_{1+1/n}^n).$$
		Since $f>0$ a.e. in $A_{1+1/n}^n$ and  $f\in L^\infty(A_{1+1/n}^n)$, we deduce $|\{x\in A_{1+1/n}^n: \nabla \varphi_1(x)=0\}|=0 $ in view of \cite[Theorem 1.1]{Lou}. Consequently, we obtain the desired conclusion.
	\end{proof}
	
	Next, we provide a result regarding the behavior of $\varphi_1$ and $\nabla \varphi_1$ at infinity which is similar to \cite[Proposition 9.1]{Alziary} but need a weaker assumption on weights. The next result is an improvement of corresponding results obtained in \cite{Anoop.CV, Chhetri-Drabek}. 
	\begin{proposition}\label{II.behavior.of.phi_1}
		Let $1<p<N$ and assume that the weight $K$ satisfies
		\begin{itemize}
			\item[(H)] $K(x)=K(|x|)>0$ for a.e. $x\in B_1^c,$ $K\in L^\infty(B_1^c)$ and $K\in L^1((1,\infty);r^\delta)$ for some $\delta\in (p-1,N-1)$.
		\end{itemize}
		Then $\varphi_1$ is radially symmetric, i.e.,  $\varphi_1(x)=\varphi_1(|x|)$ and there exists a constant $C>0$ such that
		\begin{equation}\label{II.behavior1}
		\lim_{r\to\infty}\left(r^{\frac{N-p}{p-1}}\varphi_1(r)\right)=C,
		\end{equation}
		\begin{equation}\label{II.behavior2}
		\lim_{r\to\infty}\left(r^{\frac{N-1}{p-1}}\varphi'_1(r)\right)=-\frac{N-p}{p-1}C.
		\end{equation}
	\end{proposition}
	The proof is similar to that of \cite[Proposition 9.1]{Alziary} with a little modification. For reader's convenience we sketch the proof in the Appendix~\ref{AppendixA}.
	\begin{remark}\rm
		We note that, when $1<p<N$,  $\textrm{(H)}$ implies $\textrm{(A)}.$ The important point to note here is that for our case the assumption \eqref{I.weight.RN} on the weight in \cite{Alziary} reads
		\begin{equation}\label{II.compare.K}
			0<K(x)=K(|x|)\leq \frac{C}{|x|^{p+\mu}},\quad \text{for a.e.}\  x\in B_1^c,
		\end{equation}
			with some constants $\mu>0$ and $C>0.$ Clearly, a measurable weight $K$ satisfying \eqref{II.compare.K} also satisfies $\textrm{(H)}$ with $\delta=p-1+\mu_0$ for some $\mu_0\in (0,\min\{\mu,N-p\}).$ However, the reverse is not true. 
			The following example demonstrates this fact rather strikingly. 
	\end{remark}
	\begin{example}\label{II.Example}\rm
	Let $1<p<N$ and $\zeta>1,\iota >0$. Consider 
	  \begin{align*}
	    K_1(r):= \begin{cases}
	    	\frac{1}{r^p(1+\log r)}, \quad r \in \bigcup\limits_{n=1}^{\infty} [n,n+\frac{1}{n^{\zeta}}],\\
	    	\frac{1}{r^{p+\iota}}, \quad r \in [1,\infty) \setminus (\bigcup\limits_{n=1}^{\infty}[n,n+\frac{1}{n^{\zeta}}]),
	    \end{cases}
	    \end{align*} 
	    and 
	    \begin{align*}
	    K_2(r):= \begin{cases}
	    \frac{1}{r^p}, \quad r \in \bigcup\limits_{n=1}^{\infty} [n,n+\frac{1}{n^{\zeta}}],\\
	    \frac{1}{r^{p+\iota}}, \quad r \in [1,\infty) \setminus (\bigcup\limits_{n=1}^{\infty}[n,n+\frac{1}{n^{\zeta}}]).
	    \end{cases}
	    \end{align*}
	    Then, $K(x):=K_1(|x|)$ and $\widetilde{K}(x):=K_2(|x|)$ satisfy $\textrm{(H)}$ with 
	    $\delta=p-1+\iota_0$ for $0<\iota_0<\min\{1,\iota, N-p\}$ but $K$ and $\widetilde{K}$ do not satisfy \eqref{II.compare.K}.
	\end{example}
In the rest of the paper, we always assume the weight $K$ to be admissible and denote by $(\lambda_1,\varphi_1)$ the first eigenpair of problem \eqref{II.eigen}. Define, $$X^{\bot}:=\left\{u\in X: \int_{B_1^c}K(x)\varphi_1^{p-1}u\diff x=0\right\}.$$ 
Note that, $X^{\bot}$ is a weakly closed subspace of $X,$ thanks to the compactness of the embedding $X\hookrightarrow L^p(B_1^c;w).$


	\subsection{Weighted spaces in terms of $\varphi_1$ 
	}
We introduce the following weighted spaces in terms of $\varphi_1$.  These spaces will be implemented to obtain an improved Poincar\'e inequality on $X$ when $2<p< N$ in the next section. 
	
	For $p>2$ and $\textrm{(A)}$ being satisfied,  define $\mathcal{D}_{\varphi_1}$ to be the completion of $X$ with respect to the norm
	\begin{equation*}\label{II.norm.D_phi1}
	\|u\|_{\mathcal{D}_{\varphi_1}}:=\left(\int_{B_1^c}|\nabla \varphi_1|^{p-2}|\nabla u|^2\diff x\right)^{1/2}.
	\end{equation*}
	
	We also define $\mathcal{H}_{\varphi_1}$ as the space of all measurable functions $u:\ \mathbb{R}^N\to\mathbb{R}$ such that
	\begin{equation*}\label{II.norm.H_phi1}
	\|u\|_{\mathcal{H}_{\varphi_1}}:=\left(\int_{B_1^c}K(x)\varphi_1^{p-2}u^2\diff x\right)^{1/2}<\infty.
	\end{equation*}
	Clearly, the spaces $\mathcal{D}_{\varphi_1}$ and $\mathcal{H}_{\varphi_1}$ are Hilbert spaces. Hereafter, $\textrm{(A)}$ is always assumed whenever we mention the space $\mathcal{D}_{\varphi_1}$. 
	The embeddings in the next two lemmas are crucial. The next lemma can be obtained similarly in the entire space case (see \cite[Lemma 4.3]{Alziary}).
	\begin{lemma} \label{II.le.imb}
		Assume that $p>2$. Then $X\hookrightarrow\mathcal{D}_{\varphi_1}.  $
	\end{lemma}
 The following compact embedding result 
 is proved in the Appendix~\ref{AppendixB}. 
		 
	\begin{lemma}\label{II.le.compact.imb}
		Assume that $p>2.$ Then
		$$\mathcal{D}_{\varphi_1}\hookrightarrow \mathcal{H}_{\varphi_1}\ \ \text{and}\ \ \mathcal{D}_{\varphi_1}\hookrightarrow L^2(B_1^c;|\nabla \varphi_1|^p\varphi_1^{-2}).$$ 
		If in addition, $p<N$, $\textup{(H)}$ and $\displaystyle\lim_{\rho\to\infty}\underset{r\geq\rho}{\operatorname{ess\ sup}}\ r^pK(r)=0$ hold, then the embedding $\mathcal{D}_{\varphi_1}\hookrightarrow\mathcal{H}_{\varphi_1}$ is compact.
	\end{lemma}
	\begin{remark}\rm
		We note that if a measurable weight $K$ satisfies \eqref{II.compare.K}, then it also satisfies the assumption of Lemma~\ref{II.le.compact.imb}. 
		The weight $K$ introduced in Example~\ref{II.Example} does not satisfy \eqref{II.compare.K} but fulfills the assumptions of Lemma~\ref{II.le.compact.imb}. On the other hand, the weight $\widetilde{K}$ introduced in Example~\ref{II.Example} satisfies 
		$\displaystyle\lim_{\rho\to\infty}\underset{r\geq\rho}{\operatorname{ess\ sup}}\ r^p\widetilde{K}(r)=1,$ and hence does not satisfy the assumptions of Lemma~\ref{II.le.compact.imb}.
	\end{remark} 


We now discuss differentiability of functions in $\mathcal{D}_{\varphi_1}$. For an open set $\Omega$ in $\mathbb{R}^N,$ denote by $W^1(\Omega)$ the set of all $u\in L_{loc}^1(\Omega)$ such that weak derivatives $\frac{\partial u}{\partial x_i}\  (i=1,\cdots, N)$ in $\Omega$ exist. Clearly, $X\subset W^1(B_1^c).$ The inclusion $\mathcal{D}_{\varphi_1}\subset W^1(B_1^c)$ in the case $p>2$ is not clear since the weight $|\nabla\varphi_1|^{p-2}$ is degenerate on the set $\{\nabla\varphi_1=0\}$. In the case of problem \eqref{I.eq.entire} in the entire space $\mathbb{R}^N$, the weighted space $\mathcal{D}_{\widetilde{\varphi}_1}$ (the completion of $\mathcal{D}^{1,p}(\mathbb{R}^N)$ with respect to the norm $\|u\|_{\mathcal{D}_{\widetilde{\varphi}_1}}:=\left(\int_{\mathbb{R}^N}|\nabla \widetilde{\varphi}_1|^{p-2}|\nabla u|^2\diff x\right)^{1/2}$) is not contained in $W^1(\mathbb{R}^N)$ in general. This fact can be illustrated in the following example.
\begin{example}\label{Non-weakderivative}\rm
Let $2<p<N,\ \mu>0,$ and $\gamma>\frac{(p-1)(N+2)}{p-2}-1$. Let $\widetilde{\varphi}_1$ be the corresponding positive eigenfunction of $-\Delta_p$ in $\mathbb{R}^N$ relative to the weight
$$m(x):= \begin{cases}
|x|^\gamma,& \ |x|\leq 2,\\
\frac{1}{1+|x|^{p+\mu}}, &\ |x|>2.
\end{cases}$$
Let $\phi\in C^\infty(\mathbb{R}^N)$ such that $0\leq \phi\leq 1,$ $\phi= 1$ in $B_1,$ and $\operatorname{supp}(\phi)\subset B_2.$ Let $\phi_n\in C^\infty(\mathbb{R}^N)$ such that $0\leq \phi_n\leq 1,$ $\phi_n=\phi$ in $|x|\geq \frac{2}{n},$ $\phi_n=0$ in $|x|\leq \frac{1}{n},$ and $|\nabla\phi_n|\leq 2n$ ($n=1,2,\cdots).$ Let $-\frac{(N-2)(p-1)+(\gamma+1)(p-2)}{2(p-1)}$$<\theta\leq -N$ and define $u(x):=|x|^\theta\phi(x),$ $u_n(x):= |x|^\theta\phi_n(x)$ ($n=1,2,\cdots).$ Then $u\not \in L^1_{loc}(\mathbb{R}^N),$ $\{u_n\}\subset C^1_c(\mathbb{R}^N)$ and $\displaystyle\lim_{n\to\infty}\|u_n-u\|_{\mathcal{D}_{\widetilde{\varphi}_1}}=0.$ In other words, we have $u\in\mathcal{D}_{\widetilde{\varphi}_1}\setminus W^1(\mathbb{R}^N).$  
\end{example}
In the case of an exterior domain, we still do not know whether the inclusion $\mathcal{D}_{\varphi_1}\subset W^1(B_1^c)$ is valid when $2<p<N$ and $\textrm{(H)}$ hold. However, that inclusion is guaranteed if we strengthen the assumption on $K$ as in the following lemma.
\begin{lemma}\label{II.le.weak.derivative}
	Assume that $2<p<N$ and that $\textup{(H)}$ holds. Assume in addition that 
	\begin{itemize}
		\item[(W)]  $K^{-1}\in L_{loc}^1(1,\infty)$ and for each $t>1, f(r):=\left|\int_{t}^{r}K(s)ds\right|^{\frac{2-p}{p-1}}\in L_{loc}^1(1,\infty).$
	\end{itemize}
Then, $\mathcal{D}_{\varphi_1}\subset W^1(B_1^c).$	 
\end{lemma}
The proof is provided in the Appendix~\ref{Appendix.WeakDerivative}.
\begin{example}\rm
	Clearly if the weight $K$ satisfies that $\underset{x\in A_{1+1/n}^n}{\operatorname{ess\ inf}}K(x)>0$ for all $n>2$ then $\textup{(W)}$ is satisfied. If we take $$K_3(r):=\begin{cases}
(2-r)^\eta,\quad r\in [1,2),\\
K_1(r),\quad r\in [2,\infty),
\end{cases}$$ where $K_1$ is defined in Example~\ref{II.Example} with $2<p<N$ and $0<\eta<\min\{1,\frac{1}{p-2}\}$, then the weight $K(x):=K_3(|x|)$ satisfies $\textup{(H)}$ and $\textup{(W)}$ with $\displaystyle\lim_{\rho\to\infty}\underset{r\geq\rho}{\operatorname{ess\ sup}}\ r^pK(r)=0$. On the other hand, $K$ does not satisfy \eqref{II.compare.K} and we have $\underset{r\in [1+1/n,n]}{\operatorname{ess\ inf}}K(r)=0$ for all $n>2.$
\end{example}


In the following remark we discuss the principal differences between the exterior domain case and the entire space case. 
\begin{remark}\label{II.rmk.WeakDerivative}\rm
Let  $2<p<N$ and let $\mathcal{D}_{\varphi_1}$ (resp. $\mathcal{D}_{\widetilde{\varphi}_1}$) be the weighted Sobolev space corresponding to problem \eqref{1.1} (resp.  \eqref{I.eq.entire}) with the radially symmetric and measurable weight $K$ (resp. $m$) satisfying $\textrm{(H)}$ (resp. \eqref{I.weight.RN}). The weight $|\nabla\varphi_1|^{p-2}$ of the space $\mathcal{D}_{\varphi_1}$ is degenerate on a sphere $S_{r_0}$ ($1<r_0<\infty$) whereas the weight $|\nabla\widetilde{\varphi}_1|^{p-2}$ of the space $\mathcal{D}_{\widetilde{\varphi}_1}$ is degenerate at the origin (see the proofs of Proposition~\ref{II.behavior.of.phi_1} and \cite[Proposition 9.1]{Alziary}). Although it is possible that $\mathcal{D}_{\widetilde{\varphi}_1}\not\subset W^1(\mathbb{R}^N)$ (see Example~\ref{Non-weakderivative}), we indeed have $\mathcal{D}_{\widetilde{\varphi}_1}\subset W^1(\mathbb{R}^N\setminus\{0\})$ and arguments on $\mathbb{R}^N\setminus\{0\}$ are basically the same as on $\mathbb{R}^N.$ However, the situation of an exterior domain is very different. We also have $\mathcal{D}_{\varphi_1}\subset W^1(B_1^c\setminus S_{r_0})$ thanks to the embedding $\mathcal{D}_{\varphi_1}\hookrightarrow L^2(B_1^c;|\nabla \varphi_1|^p\varphi_1^{-2})$ and the properties of $\varphi_1$ but we do not know whether $\mathcal{D}_{\varphi_1}\subset W^1(B_1^c)$. Unlike the entire space case, arguments on $B_1^c\setminus S_{r_0}$ are very different from those on $B_1^c$ since $S_{r_0}$ is a nonremovable set. So the assumption $\textrm{(W)}$ is necessary to assure that $\mathcal{D}_{\varphi_1}$ is indeed a weighted Sobolev space when $2<p<N.$ Therefore, the difference between the types of the sets where $\nabla \varphi_1$ and $\nabla \widetilde{\varphi}_1$ degenerate, makes the case of an exterior domain more delicate than the case of the entire space (see the proof in the Appendix~\ref{AppendixB} and \cite[Proof of Proposition 3.6]{Alziary}).  
\end{remark}

	
	The following operator $\mathbb{A}:\mathbb{R}^N \to \mathbb{M}_{N\times N }(\mathbb{R})$ (where $\mathbb{M}_{N\times N }(\mathbb{R}$) denotes the set of $N \times N$ matrices over $\mathbb{R}$), will provide much advantage for us when we apply the second order Taylor formula for energy functional. For $1<p<\infty,$ we define
	$$\mathbb{A}(\mathbf{a}):=|\mathbf{a}|^{p-2}\left(\mathbf{I}+(p-2)\frac{\mathbf{a}\otimes\mathbf{a}}{|\mathbf{a}|^2}\right)\ \text{for}\ \mathbf{a}\in\mathbb{R}^N\setminus\{0\},$$
	where $\mathbf{I}$ is the $N\times N$ identity matrix, $\mathbf{a}\otimes \mathbf{b}:=(a_ib_j)_{N\times N}$ with $\mathbf{a}=(a_1,\cdots,a_N),\mathbf{b}=(b_1,\cdots,b_N)\in\mathbb{R}^N.$
	We define $\mathbb{A}(\mathbf{0}):=\mathbf{0}\in \mathbb{M}_{N\times N }(\mathbb{R}).$ The following basic properties of the operator $\mathbb{A}$ were shown in \cite[Subsection 2.4]{Alziary}.
Let $1<p<\infty$, then for all $\mathbf{a},\mathbf{v}\in \mathbb{R}^N\setminus\{0\},$ we have
	\begin{equation}\label{II.est.A}
	\min\{1,p-1\}\leq \frac{\langle\mathbb{A}(\mathbf{a})\mathbf{v},\mathbf{v}\rangle_{\mathbb{R}^N}}{|\mathbf{a}|^{p-2}|\mathbf{v}|^2}\leq\max\{1,p-1\}.
	\end{equation}
Moreover, for $2\leq p<\infty$ there exists $C_p>0$  such that for all $\mathbf{a},\mathbf{b},\mathbf{v}\in\mathbb{R}^N$, we have
\begin{align*}
C_p\left(\underset{0\leq s\leq 1}{\max}|\mathbf{a}+s\mathbf{b}|\right)^{p-2}|\mathbf{v}|^2\leq \int_{0}^{1}\langle\mathbb{A}(\mathbf{a}&+s\mathbf{b})\mathbf{v},\mathbf{v}\rangle_{\mathbb{R}^N}(1-s)\diff s\notag\\
&\leq\frac{p-1}{2}\left(\underset{0\leq s\leq 1}{\max}|\mathbf{a}+s\mathbf{b}|\right)^{p-2}|\mathbf{v}|^2.         
\end{align*} 
On the other hand, for $1<p<2$ there exists $C_p>0$ such that for all $\mathbf{a},\mathbf{b},\mathbf{v}\in\mathbb{R}^N$ with $|\mathbf{a}|+|\mathbf{b}|>0$ we have
\begin{align}\label{II.est.A.int2}
\frac{p-1}{2}\left(\underset{0\leq s\leq 1}{\max}|\mathbf{a}+s\mathbf{b}|\right)^{p-2}|\mathbf{v}|^2\leq \int_{0}^{1}\langle\mathbb{A}(\mathbf{a}&+s\mathbf{b})\mathbf{v},\mathbf{v}\rangle_{\mathbb{R}^N}(1-s)\diff s\notag\\
&\leq C_p\left(\underset{0\leq s\leq 1}{\max}|\mathbf{a}+s\mathbf{b}|\right)^{p-2}|\mathbf{v}|^2.
\end{align} 
By \eqref{II.est.A}, when 
$p> 2$, we have
\begin{equation}\label{II.est.A-norm}
\|v\|^2_{\mathcal{D}_{\varphi_1}}\leq \int_{B_1^c}\langle\mathbb{A}(\nabla\varphi_1)\nabla v,\nabla v\rangle_{\mathbb{R}^N}\diff x\leq(p-1)\|v\|^2_{\mathcal{D}_{\varphi_1}},\ \forall v\in\mathcal{D}_{\varphi_1}.
\end{equation}

\section{ An Improved Poincar\'e inequality when $2<p<N$}\label{Poincare.Inequality}

In this section, we obtain an improved Poincar\'e inequality on $X$ when $2<p< N,$ by applying the second order Taylor formula for energy functional at $\varphi_1$. 

For functions $\phi,v,w: B_1^c\to\mathbb{R}$, we define 
\begin{align}\label{II.def.Q}
\mathcal{Q}_\phi(v,w):=&\int_{B_1^c}\left\langle\left[\int_{0}^{1}\mathbb{A}(\nabla\varphi_1+s\nabla\phi)(1-s)\diff s\right]\nabla v,\nabla w\right\rangle_{\mathbb{R}^N}\diff x\notag\\
&-\lambda_1(p-1)\int_{B_1^c}K(x)\left[\int_{0}^{1}|\varphi_1+s\phi|^{p-2}(1-s)\diff s\right]vw\diff x,   \notag
\end{align}
and thus
$$\mathcal{Q}_0(v,v)=\frac{1}{2}\int_{B_1^c}\langle\mathbb{A}(\nabla\varphi_1)\nabla v,\nabla v\rangle_{\mathbb{R}^N}\diff x-\frac{1}{2}\lambda_1(p-1)\int_{B_1^c}K(x)\varphi_1^{p-2}v^2\diff x,$$
whenever the integrals are well-defined. Note that when $p\geq 2$, by invoking the Lebesgue dominated convergence theorem, we can show that, the functional  
$$\Phi (u):=\frac{1}{p}\int_{B_1^c}|\nabla u|^p\diff x-\frac{\lambda_1}{p}\int_{B_1^c}K(x)|u|^p\diff x,$$
belongs to $C^2(X,\mathbb{R})$ via standard arguments. Applying the second order Taylor formula for $\Phi$ at $\varphi_1$, we have
$$\Phi(\varphi_1+\phi)= \Phi(\varphi_1)+\langle D\Phi(\varphi_1),\phi\rangle+\int_0^1(1-s)\langle D^2\Phi(\varphi_1+s\phi)\phi,\phi\rangle \diff s, \quad \forall \phi\in X.$$ 
Thus,  
$\Phi(\varphi_1+\phi)=\mathcal{Q}_\phi(\phi,\phi)$ and hence, $\mathcal{Q}_\phi(\phi,\phi)\geq 0$  for all $\phi\in X$ due to variational characterization of the first eigenvalue $\lambda_1.$ 
Clearly, $\mathcal{Q}_0(\varphi_1,\varphi_1)=0$. When  $p> 2$, $\mathcal{Q}_0(\cdot,\cdot)$ is well-defined on $\mathcal{D}_{\varphi_1}$ in view of \eqref{II.est.A-norm} and the embedding $\mathcal{D}_{\varphi_1}\hookrightarrow\mathcal{H}_{\varphi_1}.$ Arguing as in \cite[the inequality $(4.4)$]{Takac2002}, we get  $\mathcal{Q}_0(\phi,\phi)\geq 0$ for all $\phi\in\mathcal{D}_{\varphi_1}.$ So, we obtain
\begin{lemma}\label{nonnegativeness.of.Q0}
	Assume that $p> 2.$ Then, $\mathcal{Q}_0(\varphi_1,\varphi_1)=0$ and $0\leq \mathcal{Q}_0(\phi,\phi)<\infty$ for all $\phi\in\mathcal{D}_{\varphi_1}.$
\end{lemma}
By Lemma~\ref{nonnegativeness.of.Q0} we have another formula for the first eigenvalue
\begin{equation}\label{2nd.form.of.lamda1}
\lambda_1=\inf\left\{\frac{\int_{B_1^c}\langle\mathbb{A}(\nabla\varphi_1)\nabla u,\nabla u\rangle_{\mathbb{R}^N}\diff x}{(p-1)\int_{B_1^c}K(x)\varphi_1^{p-2}u^2\diff x}:\ u\in\mathcal{D}_{\varphi_1}\setminus\{0\}\right\},
\end{equation}
and $\varphi_1$ is a minimizer for $\lambda_1$ in \eqref{2nd.form.of.lamda1}. Clearly, $u$ is a minimizer for $\lambda_1$ in \eqref{2nd.form.of.lamda1} if and only if $u\in\mathcal{D}_{\varphi_1}\setminus\{0\}$ and $\mathcal{Q}_0(u,u)=0.$ This is equivalent to $u\in\mathcal{D}_{\varphi_1}\setminus\{0\}$ and $\mathcal{Q}_0(u,v)=0$ for all $v\in\mathcal{D}_{\varphi_1}$ since $\mathcal{Q}_0(\cdot,\cdot)$ is a nonnegative symmetric bilinear form on $\mathcal{D}_{\varphi_1}.$ Hence, if $\mathcal{D}_{\varphi_1}\subset W^1(B_1^c)$ then  $u\in\mathcal{D}_{\varphi_1}\setminus\{0\}$ satisfies $\mathcal{Q}_0(u,u)=0$ if and only if $u\in\mathcal{D}_{\varphi_1}$ is nontrivial weak solution in $\mathcal{D}_{\varphi_1}$ to problem
\begin{equation}\label{II.simplicity.linear}
-\nabla\cdot(\mathbb{A}(\nabla\varphi_1)\nabla u)=\lambda_1(p-1)K\varphi_1^{p-2}u\ \ \text{in}\ B_1^c.
\end{equation}
In other words, $u$ is an eigenfunction associated with the first eigenvalue $\lambda_1$ of \eqref{II.simplicity.linear}. 
The following result shows that $\lambda_1$ is in fact a simple eigenvalue of \eqref{II.simplicity.linear} when $2<p<N.$

\begin{proposition}\label{II.prop.simplicity}
	Let $2<p<N$, $\textup{(H)}$ and $\textup{(W)}$ hold. Then a function $u\in\mathcal{D}_{\varphi_1}$ satisfies $\mathcal{Q}_0(u,u)=0$ if and only if $u=k\varphi_1$ for some constant $k\in \mathbb{R}.$
\end{proposition}


\begin{remark}\rm\label{II.rmk.proof_of_simplicity}
	The simplicity of the first eigenvalue $\lambda_1$ of degenerated linear problem~\eqref{II.simplicity.linear} is a by-product of our work which is of independent interest. The analogue for the entire space case is dealt with in \cite[Proposition 5.2 and its proof]{Alziary}. Our case, which is more delicate due to the arguments presented in Remark~\ref{II.rmk.WeakDerivative}, is proved in detail in the Appendix~\ref{AppendixC}.
\end{remark}
We close this section with the following \textit{improved Poincar\'e inequality} on $X$ when $2<p<N$. The proof 
is almost similar to that of a bounded domain case \cite[Theorem 1.1]{Fleckinger-Takac} and the entire space case \cite[Lemma 3.7]{Alziary}. It has not escaped our notice that no restriction either on $K$ or $N$ is required for the linear case $p=2$, which we include for completeness. 
\begin{proposition}\label{II.Prop.Poincare}
	\begin{itemize}
		\item[(i)]
	Let $p=2$. Then there exists $C=C(K)>0$ such that 
	\begin{align}\label{IV.Poincarep=2}
	\int_{B_1^c}|\tau\nabla \varphi_1+\nabla u^{\bot}|^2\diff x-\lambda_1\int_{B_1^c}K(x)|\tau\varphi_1+u^{\bot}|^2\diff x
	\geq C \int_{B_1^c}|\nabla u^{\bot}|^2\diff x,
	\end{align}
	holds for all $\tau\in\mathbb{R}$ and $u^{\bot} \in X^{\bot}.$
	\item[(ii)] Let $2<p<N$, $\textup{(H)}$, $\textup{(W)}$ and $\displaystyle \lim_{\rho\to\infty}\underset{r\geq\rho}{\operatorname{ess\ sup}}\ r^pK(r)=0$ hold. Then there exists $C=C(p,K)>0$ such that
		\begin{align}\label{IV.Poincare}
	\int_{B_1^c}|\tau\nabla \varphi_1+\nabla u^{\bot}|^p\diff x&-\lambda_1\int_{B_1^c}K(x)|\tau\varphi_1+u^{\bot}|^p\diff x\notag\\
	&\geq C\left(|\tau|^{p-2}\int_{B_1^c}|\nabla \varphi_1|^{p-2}|\nabla u^{\bot}|^2\diff x+\int_{B_1^c}|\nabla u^{\bot}|^p\diff x\right),
	\end{align}
	holds for all $\tau\in\mathbb{R}$ and $u^{\bot} \in X^{\bot}.$ 
\end{itemize}
\end{proposition}
\begin{proof}
To prove part (i), i.e., the linear case $p=2$, we use the variational characterization of the second eigenvalue  
	\begin{equation*}
		\lambda_2:=\inf\left\{\int_{B_1^c}|\nabla u|^2\diff x: u\in X^\bot, \int_{B_1^c} K(x)|u|^2\diff x=1 \right\}>\lambda_1,
		\end{equation*}
to obtain \eqref{IV.Poincarep=2} with $C=\frac{\lambda_2-\lambda_1}{\lambda_2}.$
In order to prove part (ii), we can use 
the embeddings of $\mathcal{D}_{\varphi_1}$ and the properties of $\mathcal{Q}_\phi(\cdot,\cdot)$, whenever the assumptions are satisfied. Since the proof is almost identical to that of \cite[Lemma 3.7]{Alziary}, we omit it.

\end{proof}


\section{A saddle point geometry when $1<p<2$}\label{Saddle.point.property}
Let us consider the energy functional associated with problem \eqref{1.1} with $\lambda=\lambda_1$ (resonant case), 
\begin{equation}
 J_h(u)=\frac{1}{p}\int_{B_1^c}|\nabla u|^p\diff x-\frac{\lambda_1}{p}\int_{B_1^c}K(x)|u|^p\diff x-\langle h,u\rangle.     \label{energyfnl}   \notag
\end{equation}

The following notion introduced in \cite{Dra-Hol} will play an important role.
\begin{definition}\rm\label{III.def.saddle_point_geometry}
	We say that $J_h:X \to \mathbb{R}$ has a \emph{saddle point geometry}, if there exist $u,v\in X$, such that
	$$\int_{B_1^c}K(x)\varphi_1^{p-1}u\diff x<0<\int_{B_1^c}K(x)\varphi_1^{p-1}v\diff x\quad\ \text{and}$$
	$$\max\{J_h(u), J_h(v)\}<\underset{w\in X^{\bot}}{\inf}J_h(w).$$
\end{definition}
In a ball or in the entire space case, a saddle point geometry for the energy functional occurs, when the source term $h$ satisfies $h\not\equiv0$ and $\langle h,\varphi_1\rangle=0$. 
 The authors in \cite{Dra-Hol, Alziary} used the second order Taylor formula for the energy functional at $\varphi_1$, to prove this fact. Likewise we expect, there is a  $\phi$ satisfying the condition
\begin{itemize} 
	\item[$\mathrm{(P_h)}$] $\phi\in C_c^1(B_1^c)$, $\phi$ is constant on a neighbourhood of 
	$\mathcal{A}=\{x\in B_1^c: \nabla\varphi_1(x)=0\}$
	and satisfies $\langle h,\phi\rangle\ne 0$.
\end{itemize}  
However, unlike a ball or the entire space case, in the exterior domain case there exists $h \in X^\ast \setminus \{0\}$, such that $\langle h,\varphi_1\rangle =0$ and there is no $\phi$ satisfying $\mathrm{(P_h)}$, even if $K$ is of a special form. This interesting fact is stated in the following result.


\begin{proposition}\label{III.prop.h}
	Let $1<p<N$  and $\textup{(H)}$ be satisfied. Then there is an $h\in X^\ast\setminus\{0\}$ such that $\langle h,\varphi_1\rangle= 0$ and $h\equiv 0$ on the set
	$$Y:=\{\phi \in C_c^1(B_1^c): \phi \ \text{is constant on some neighbourhood of}\  \mathcal{A}\}.$$
\end{proposition}
Note that for $K$ satisfying $\textrm{(H)}$, we have $\mathcal{A}=S_{r_0}$ for some $r_0>1.$ In order to prove Proposition~\ref{III.prop.h}, we first use the positivity of the capacity of 
 $\Gamma:=S_{r_0}\cap B_{\mathbb{R}^N}(x_0,r)$, for some $x_0\in S_{r_0}$ and $0<r<2r_0$ to show that


\begin{lemma}\label{III.le.dense}
	Under the assumption of Proposition~\ref{III.prop.h}, the set $Y$ is not dense in $X$.
\end{lemma}
\begin{proof}
	Assume to the contrary that $Y$ is dense in $X$. Let $u\in C_c^1(B_1^c)$ be nonconstant on $S_{r_0}$, i.e., $M:=\displaystyle\max_{x\in S_{r_0}}u(x)> m:=\min_{x\in S_{r_0}}u(x)$. 
	The density of $Y$ implies that, there exists a sequence $\{u_n\}\subset Y$, such that $u_n\to u$ in $X$ as $n\to \infty.$	
	Since $u_n \in Y$, there exists $c_n \in \mathbb{R}$, satisfying $u_n\equiv c_n$ in a neighbourhood $\mathcal{N}_n $ of $S_{r_0}$. We claim that $c_n\to M$ as $n\to\infty.$ If this is not the case then there is a subsequence of $\{c_n\}$ (still denoted by $\{c_n\}$) and some $\epsilon>0$ such that
	$$|c_n-M|>\epsilon, \quad \forall n\in \mathbb{N}.$$
	This yields, up to a subsequence, $c_n-M>\epsilon$ for all $n\in \mathbb{N}$ or $M-c_n>\epsilon$ for all $n\in \mathbb{N}.$
	
	Suppose that $c_n-M>\epsilon$ for all $n\in \mathbb{N}.$ Now by the definition of $M$ and the continuity of $u$, there is a $\delta\in (0,r_0-1)$ such that $u(x)< M+\frac{\epsilon}{2}$ for all $x\in A_{r_0-\delta}^{r_0+\delta}$. For each $n$, set $w_n:=\frac{3}{\epsilon}|u_n-u|$ and also set $W_n:=\mathcal{N}_n\cap A_{r_0-\delta}^{r_0+\delta}$. Then $w_n\geq 0,$ $w_n\in L^{\frac{Np}{N-p}}(\mathbb{R}^N)$ and $|\nabla w_n|\in L^p(\mathbb{R}^N).$ Moreover, for all $x\in W_n$, we have
	$$w_n(x) >\frac{3}{\epsilon}\left(c_n-M-\frac{\epsilon}{2}\right)>\frac{3}{2}.$$
	So by the definition of $p$-capacity (see \cite[Definition 4.7.1]{Evans}) and the fact that 
	$w_n\to 0$ in $X$ as $n\to\infty$, we obtain
	$$\operatorname{Cap}_p(S_{r_0})
	\leq \int_{\mathbb{R}^N}|\nabla w_n|^p\diff x=\|w_n\|^p \to 0\ \text{as}\ n\to \infty. $$
	But $\operatorname{Cap}_p(S_{r_0})=0$ contradicts \cite[Application B of Subsection 3.3.4 and Theorem 4 of Section 4.7]{Evans}.
	
	We now consider the other case, $M-c_n>\epsilon$ for all $n\in\mathbb{N}.$ Let $x_M\in S_{r_0}$ be such that $u(x_M)=M.$ By the continuity of $u$ again, there is a $\delta\in (0,r_0-1)$ such that $u(x)>M-\frac{\epsilon}{2}$ for all $x\in B_{\mathbb{R}^N}(x_M,\delta)$. Set $\Gamma:=B_{\mathbb{R}^N}(x_M,\frac{\delta}{2})\cap S_{r_0}$ and for each $n,$ set $w_n:=\frac{3}{\epsilon}|u-u_n|$ and $W_n:=\mathcal{N}_n\cap B_{\mathbb{R}^N}(x_M,\delta)$. Then for all $x\in W_n$, we have
	$$w_n(x) >\frac{3}{\epsilon}\left(M-\frac{\epsilon}{2}-c_n\right)>\frac{3}{2}.$$
	Arguing as in the previous case we obtain $\operatorname{Cap}_p(\Gamma)=0$, which contradicts \cite[Application B of Subsection 3.3.4 and Theorem 4 of Section 4.7]{Evans}.
	
	From the arguments above we obtain that $c_n\to M$ as $n\to\infty.$ Then there exist $n_0\in\mathbb{N}$ and $\epsilon'>0$ such that
	\begin{equation*}
	c_n>m+\epsilon', \quad \forall n\geq n_0.
	\end{equation*}
	Let $x_m\in S_{r_0}$ be such that $u(x_m)=m.$ By the continuity of $u$ there is a $\delta'\in (0,r_0-1)$ such that $u(x)<m+\frac{\epsilon'}{2}$ for all $x\in B_{\mathbb{R}^N}(x_m,\delta')$. Set $\Gamma':=B_{\mathbb{R}^N}(x_m,\frac{\delta'}{2})\cap S_{r_0}$ and for each $n,$ set $w_n':=\frac{3}{\epsilon'}|u_n-u|$ and $W_n':=\mathcal{N}_n\cap B_{\mathbb{R}^N}(x_m,\delta').$ Then for all $n\geq n_0$ and all $x\in W_n'$ we have
	$$w_n'(x) >\frac{3}{\epsilon'}(c_n-m-\frac{\epsilon'}{2})>\frac{3}{2}.$$
		Proceeding as before, we get $\operatorname{Cap}_p(\Gamma')=0$, which is again a contradiction. The proof of Lemma~\ref{III.le.dense} is complete.

\end{proof}

The next result shows that $\varphi_1$ 
belongs to the closure of $Y$ in $X$.
\begin{lemma}\label{III.le.phi1.in.closure.of.Y}
	Under the assumption of Proposition ~\ref{III.prop.h}, 
	we have $\varphi_1\in \overline{Y}.$
\end{lemma}

\noindent The proof of this lemma can be found in the Appendix~\ref{AppendixD}.

Invoking Lemma~\ref{III.le.dense}, Lemma~\ref{III.le.phi1.in.closure.of.Y} and the Hahn-Banach Theorem we prove Proposition~\ref{III.prop.h}. 

\begin{proof}[Proof of Proposition~\ref{III.prop.h}]
	By Lemma~\ref{III.le.dense}, there exists a $\phi_0\in X\setminus\overline{Y}.$ Note that $\overline{Y}$ is a closed linear subspace of $X$. Define $$g:\  W:=\operatorname{span}\{\phi_0\}\oplus\overline{Y}\to \mathbb{R},$$
	$$ g(t\phi_0+v)=t,\quad \forall  (t,v) \in \mathbb{R} \times \overline{Y}.$$
	It is easy to see that $g$ is linear and there is a positive constant $C$, such that
	\begin{equation}\label{III.ineq1}
	g(t\phi_0+v)\leq C\|t\phi_0+v\|,\quad \forall (t,v) \in \mathbb{R} \times \overline{Y}.
	\end{equation}
	To prove this claim we first show that
	\begin{equation}\label{III.ineq2}
	\inf_{v\in \overline{Y}}\|\phi_0+v\|>0.
	\end{equation}
	Indeed if \eqref{III.ineq2} is not true then there is a sequence $\{v_n\}\subset \overline{Y}$ such that $\phi_0+v_n\to 0$ in $X$ as $n\to\infty$. This leads to $\phi_0\in\overline{Y},$ a contradiction. So we obtain \eqref{III.ineq2}. We now return to prove \eqref{III.ineq1}. Let $C=\left(\inf_{v\in \overline{Y}}\|\phi_0+v\|\right)^{-1}.$ The case $t\leq 0$ is trivial. For $t>0$, \eqref{III.ineq1} is equivalent to
		\begin{equation*}
	g(\phi_0+\frac{1}{t}v)\leq C\|\phi_0+\frac{1}{t}v\|,\quad \forall v \in \overline{Y},
	\end{equation*}
	i.e.,
		\begin{equation*}
1\leq C\|\phi_0+\frac{1}{t}v\|,\quad \forall v \in \overline{Y}.
	\end{equation*}
That holds true by the choice of $C$ and hence, \eqref{III.ineq1} is proved. Next, invoking the Hahn-Banach Theorem we can extend $g$ to a linear functional $h: X\to \mathbb{R}$ such that $h|_W=g$ and
	$|h(u)|\leq C\|u\|$  for all $u\in X$. Thus we can find an $h\in X^\ast$ such that $h(\phi_0)=1$ and $h|_{\overline{Y}}=0$ and this completes the proof of Proposition~\ref{III.prop.h} in view of Lemma~\ref{III.le.phi1.in.closure.of.Y}.
\end{proof}
\begin{remark}\rm\label{III.rmk.compare}
	Proposition~\ref{III.prop.h} illustrates another significant difference between the problem in the entire space $\mathbb{R}^N$ studied in \cite{Alziary} and the problem in $B_1^c$. Indeed, in the entire space case, $\mathcal{A}=\{0\}$ and for all $h\in X^\ast\setminus\{0\}$ we can find $\phi\in C_c^1(\mathbb{R}^N)$ such that $\langle h,\phi\rangle=1$ and $0\not\in \operatorname{supp}(\phi)$ (see \cite[Proof of Lemma 3.9]{Alziary}).
\end{remark}

To find an optimal condition on  $h\in X^\ast\setminus\{0\}$ such that there is a $\phi$ satisfying $\mathrm{(P_h)},$ we introduce the condition: 
$$h\in X^\ast_Y,\ \text{where}\ X^\ast_Y:=\{h\in X^\ast: h\not \equiv 0\ \text{on}\ Y\}.$$

\noindent This condition is reasonable due to the following result.


\begin{lemma}\label{III.set.of.h}
	Assume that the admissible weight $K$ satisfies $K>0$ a.e. in $B_1^c$ and $K\in L^\infty(A_1^R)$ for all $R>1.$ Then $X^\ast_Y$ contains the set
	$$Z:=\left\{h\in X^\ast\setminus\{0\}:\ \exists g\in C_c(B_1^c)\ \text{such that}\ \langle h,u\rangle =\int_{B_1^c}gu\diff x, \forall u\in X\right\},$$
	and $ X^\ast_Y$ is open and dense in $X^\ast$.
\end{lemma}
We emphasize that, the embedding $X\hookrightarrow L_{loc}^p(B_1^c)$ implies that $u\mapsto \int_{B_1^c}gu\diff x$ is a linear bounded functional on $X$ for each $g\in C_c(B_1^c).$ The assumption $K\in L^\infty(A_1^R)$ for all $R>1$ guarantees that $\operatorname{dist}(\mathcal{A},\partial B_1)>0$ and hence, $Y\ne\emptyset.$
\begin{proof}[Proof of Lemma~\ref{III.set.of.h}]
	
	Suppose that $h\in X^\ast\setminus\{0\}$ and $\langle h,u\rangle =\int_{B_1^c}gu\diff x$ for all $u\in X$ for some $g\in C_c(B_1^c)$. Since $h\ne 0$ then so is $g$ and hence, $g(x_0)\ne 0$ for some $x_0\in B_1^c.$ By the       continuity of $g$, there is $r_0\in (0, |x_0|-1)$ such that $g(x)g(x_0)>0$ for all $x\in B_{\mathbb{R}^N}(x_0,r_0).$ By Lemma~\ref{II.meas.}, we have $\operatorname{int}(\mathcal{A})=\emptyset$. Thus, there exists $ x_1\in B_{\mathbb{R}^N}(x_0,r_0)\setminus \mathcal{A}$. Since $K\in L^\infty(A_1^R)$ for all $R>1$, $\mathcal{A}$ is closed and $\operatorname{dist}(\mathcal{A},\partial B_1)>0.$ Therefore, there exists $r_1\in (0,r_0-|x_1-x_0|)$ such that $B_{\mathbb{R}^N}(x_1,r_1)\cap \mathcal{A}=\emptyset.$ Then let $\phi\in C^\infty(\mathbb{R}^N)$ be such that $0\leq\phi \operatorname{sgn}(g(x_0))\leq 1,$ $\phi\equiv \operatorname{sgn}(g(x_0))$ in $B_{\mathbb{R}^N}(x_1,r_1/2)$ and $\phi\equiv 0$ in $\mathbb{R}^N \setminus B_{\mathbb{R}^N}(x_1,r_1).$
	Thus, $\phi\in Y$ and $\langle h,\phi\rangle >0$ and hence $h\in X^\ast_Y,$ i.e., $Z\subset X^\ast_Y.$
	
	Since $Z\subset X^\ast_Y$, to prove $\overline{X^\ast_Y}=X^\ast$ it suffices to show that $\overline{Z}=X^\ast$. Before we proceed further, we first observe that by a similar argument to that of \cite[Proof of Proposition 8.14]{Brezis-book}, we obtain that for a given $h\in X^\ast$, there exist $g_1,\cdots,g_N\in L^{p'}(B_1^c)$ with $\frac{1}{p}+\frac{1}{p'}=1$, such that
	\begin{equation}\label{III.form.of.h}
	\langle h,u\rangle=\sum_{i=1}^{N}\int_{B_1^c}g_i\frac{\partial u}{\partial x_i}\diff x.
	\end{equation}
	Next, let $h\in X^\ast$ be of the form \eqref{III.form.of.h}. For any given $\epsilon>0$, by the density of $C_c^\infty(B_1^c)$ in the $L^{p'}(B_1^c)$ for each $i\in\{1,\cdots, N\}$ we find $\widetilde{g}_i\in C_c^\infty(B_1^c)$ such that
	$$\|\widetilde{g}_i-g_i\|_{p'}\leq \frac{\epsilon}{N},$$
	where $\|\cdot\|_q$ denotes the usual Lebesgue norm on $L^q(B_1^c)$ $(1<q<\infty).$ Then for $\widetilde{h}\in X^\ast$, given by
	$$\langle \widetilde{h},u\rangle=\sum_{i=1}^{N}\int_{B_1^c}\widetilde{g}_i\frac{\partial u}{\partial x_i}\diff x,$$
	we have
	\begin{equation*}
	|\langle\widetilde{h}-h,u\rangle|= \left|\sum_{i=1}^{N}\int_{B_1^c}(\widetilde{g}_i-g_i)\frac{\partial u}{\partial x_i}\diff x\right|\leq \sum_{i=1}^{N}\|\widetilde{g}_i-g_i\|_{p'}\|\frac{\partial u}{\partial x_i}\|_p\leq \epsilon\|u\|,\ \forall u\in X.
	\end{equation*}
	Thus, $\|\widetilde{h}-h\|_{X^\ast}\leq \epsilon$ and note that $\langle \widetilde{h},u\rangle=\int_{B_1^c}\left(-\sum_{i=1}^{N}\frac{\partial \widetilde{g}_i}{\partial x_i}\right)u\diff x=\int_{B_1^c}\widetilde{g}u\diff x,$
	where $\widetilde{g}:=-\sum_{i=1}^{N}\frac{\partial \widetilde{g}_i}{\partial x_i}\in C_c(B_1^c).$ This implies the density of $Z$ in $X^\ast.$
	Finally, we show that $X^\ast_Y$ is open in $X^\ast.$ If this is not the case then there is an $h\in X^\ast_Y$ and a sequence $\{h_n\}\subset X^\ast$ with $h_n\equiv 0$ on $Y$ such that $\|h_n-h\|_{X^\ast}<\frac{1}{n}.$ Let $\phi\in Y$ be such that $\langle h,\phi\rangle\ne 0$. Then, we have
	$$|\langle h,\phi\rangle|=|\langle h_n-h,\phi\rangle|\leq\frac{1}{n}\|\phi\|\to 0\ \text{as}\ n\to\infty,$$
	a contradiction. So the proof is complete.
	\end{proof}

The following proposition together with the fact that $J_h$ is bounded from below on $X^{\bot}$ (this will be shown in the next section) provide a saddle point geometry of the energy functional associated with problem \eqref{1.1} in the resonant case.
\begin{proposition}\label{III.Prop.saddle}
	Assume that $1<p<2$ and that $K\in L^\infty(A_1^R)$ for all $R>1$. Let $h\in X^\ast_Y$ with $\langle h,\varphi_1\rangle=0$. Then for any $M>0$, there exist $\tau_0>0$, such that for each $\tau>\tau_0$ we can find $v^\tau_\pm\in X^{\bot}$ such that 
	$$\max\{J_h(\tau\varphi_1+v^\tau_+),J_h(-\tau\varphi_1+v^\tau_-)\}<-M.$$
	
\end{proposition}
\begin{proof}
	 Let $M>0$ be given. Let $\phi\in Y$ such that $\langle h,\phi\rangle =1.$ Set
	$$u_{\pm}:=\pm t\varphi_1+t^{\frac{2-p}{2}}\phi\ \text{with}\ t>0.$$
	Since $\phi\in Y$, $\phi$ is constant on a neighbourhood of 
	$\{x\in B_1^c: \nabla\varphi_1(x)=0\}$ and has a compact support in $B_1^c.$ So by using the Lebesgue dominated convergence theorem, we can show that $f(t):=\frac{1}{p}\int_{B_1^c}|\nabla\varphi_1+t\nabla\phi|^p\diff x-\frac{\lambda_1}{p}\int_{B_1^c}K(x)|\varphi_1+t\phi|^p\diff x$ belongs to $C^2[-t_0,t_0]$ for some $0<t_0\ll 1$. For each $\xi\in [-t_0,t_0],$ applying the second order Taylor formula for function $s\mapsto f(\xi s)$ and utilizing the properties of $(\lambda_1,\varphi_1)$ we get
	$$f(\xi)=\int_{0}^{1}\xi^2f''(\xi s)(1-s)ds.$$
	Thus, for sufficiently large $t$, we have
	\begin{equation}\label{III.J.behavior}
	J_h(u_{\pm})=t^pf\left(\pm t^{-\frac{p}{2}}\right)-t^{\frac{2-p}{2}}=\mathcal{Q}_{\pm t^{-\frac{p}{2}}\phi}(\phi,\phi)-t^{\frac{2-p}{2}},
	\end{equation}
	where
	\begin{align*}\label{III.Q}
	\mathcal{Q}_{\pm t^{-\frac{p}{2}}\phi}(\phi,\phi)=&\int_{B_1^c}\left\langle\left[\int_{0}^{1}\mathbb{A}\left(\nabla\varphi_1+s(\pm t^{-\frac{p}{2}}\nabla\phi)\right)(1-s)\diff s\right]\nabla \phi,\nabla \phi\right\rangle_{\mathbb{R}^N}\diff x\notag\\
	&-\lambda_1(p-1)\int_{B_1^c}K(x)\left[\int_{0}^{1}|\varphi_1+s(\pm t^{-\frac{p}{2}}\phi)|^{p-2}(1-s)\diff s\right]\phi^2\diff x.
	\end{align*}
	Using \eqref{II.est.A.int2}, we have
	\begin{align*}
	\int_{0}^{1}&\left\langle\mathbb{A}\left(\nabla\varphi_1+s(\pm t^{-\frac{p}{2}}\nabla\phi)\right)\nabla \phi,\nabla \phi\right\rangle_{\mathbb{R}^N}(1-s)\diff s\\
	&\leq C_p\left(\underset{0\leq s\leq 1}{\max}|\nabla\varphi_1+s(\pm t^{-\frac{p}{2}}\nabla\phi)|\right)^{p-2}|\nabla\phi|^2\leq C_p|\nabla\varphi_1|^{p-2}|\nabla\phi|^2.
	\end{align*}
	On the other hand, for $x\in \operatorname{supp}(\phi)$, $s\in [0,1]$ and $t>0$ large, we have	
	$$|\varphi_1+s(\pm t^{-\frac{p}{2}}\phi)|\geq \inf_{\operatorname{supp}(\phi)}\varphi_1-t^{-\frac{p}{2}}\|\phi\|_\infty>\frac{1}{2}\inf_{\operatorname{supp}(\phi)}\varphi_1>0.$$
The last two estimates imply that, there exist $t_\phi>0$ and $\widetilde{M}>0$, such that
	$$\mathcal{Q}_{\pm t^{-\frac{p}{2}}\phi}(\phi,\phi)<\widetilde{M},\ \forall t>t_\phi.$$
	Combining this with \eqref{III.J.behavior}, we get
	$$J_h(\pm t\varphi_1+t^{\frac{2-p}{2}}\phi)<-M,$$
	for all $t>t_1:=\max\left\{t_\phi,(M+\widetilde{M})^{\frac{2}{2-p}}\right\}.$ 
	Decompose $\phi=\tau_\phi\varphi_1+\phi^{\bot}$ with $\tau_\phi\in\mathbb{R},\phi^{\bot}\in X^{\bot}$ and consider $g_\pm(t):=\pm t+t^{\frac{2-p}{2}}\tau_\phi.$ Clearly, $g_\pm$ are continuous in $(0,\infty)$; moreover, for $t > t_2:=\left(\frac{(2-p)|\tau_\phi|}{2}\right)^{\frac{2}{p}}$, we see that
	$g_+$ is strictly increasing while $g_-$ is strictly decreasing. 
	Let $t_3>\max\{t_1,t_2\}$ such that $g_-(t_3)<0.$ Set $\tau_0:=\max\{g_+(t_3),-g_-(t_3)\}.$ Thus, for any $\tau>\tau_0$ let $t_\pm>t_3$ be such that $\pm\tau=g_\pm(t_\pm)$ and $v^\tau_\pm=t_\pm^{\frac{2-p}{2}}\phi^{\bot}\in X^{\bot}.$ Then, we have
	$$J_h(\pm\tau\varphi_1+v_\pm^\tau)=J_h(\pm t_\pm\varphi_1+t_\pm^{\frac{2-p}{2}}\phi)<-M,$$ and this finishes the proof.
	\end{proof}


\section{Existence and multiplicity results}\label{Sec.Existence}

\subsection{Statements of the existence results} In this section we investigate the existence and multiplicity of solutions of problem \eqref{1.1}. In the non-resonant case, Anoop et al. \cite{Anoop.NA} obtained the following existence result thanks to the isolatedness of $\lambda_1$ and the Fredholm alternative for the $p$-Laplacian due to Fu\v{c}\'{i}k et al. \cite[Chapter II, Theorem 3.2]{Fucik} (see also \cite[Theorem 4.4]{Stavrakakis}). 
Recall that $K$ is always assumed to be admissible.


\begin{theorem}[cf. {\cite[Proposition 3.1]{Anoop.NA}}]
	Let $p>1.$ 
	Then for every $\lambda\in (0,\lambda_1)$ and $h\in X^\ast$, problem \eqref{1.1} admits a solution in $X$. If in addition $K\in L^s(A_1^R)\cap L_{loc}^\infty(B_1^c)$ for some $s>\frac{N}{p}$ and $R>1$ when $1<p\leq N$ or $K\in L_{loc}^\infty(B_1^c)$ when $p>N$ then $\lambda_1$ is isolated and there exists $\delta>0$ such that for every $\lambda\in (\lambda_1,\lambda_1+\delta)$ and $h\in X^\ast$, problem \eqref{1.1} also admits a solution in $X$.
\end{theorem}
In the resonant case for the linear problem, i.e., $\lambda=\lambda_1$ and $p=2,$ it is easy to see that a necessary condition for solvability of problem \eqref{1.1} is $\langle h,\varphi_1\rangle=0.$ As we will see later, it is also a sufficient condition in this case (see Theorem~\ref {V.Theo.degenerate} (i) below). We will also see that this condition is not necessary but sufficient for existence of a solution for any $p>1$, $p \neq 2.$ More precisely, we obtain existence and multiplicity of solutions to problem \eqref{1.1}, for the resonant case with $h$ in a neighbourhood of some $h^\ast\in X^\ast\setminus\{0\}$ satisfying $\langle h^\ast,\varphi_1\rangle=0$, by modifying variational arguments used in \cite{Alziary,Drabek.EJDE2002}. As seen in Proposition~\ref{III.Prop.saddle} and Proposition~\ref{II.Prop.Poincare}, for $\lambda=\lambda_1$ and $h=h^\ast$ the energy functional corresponding to problem \eqref{1.1} is unbounded from below in case $1<p<2$, whereas it is bounded from below in case $2< p<N.$ Because of this we will deal with the singular case $1<p<2$ and the degenerate case $2<p<N$ separately. 
We first state our main result for the singular case $1<p<2$. 


\begin{theorem}\label{V.Theo.singular}
	Let $1<p<2$ and let $K\in L^\infty(A_1^R)$ for all $R>1.$ Let $h^\ast\in X^\ast_Y$ be such that $\langle h^\ast,\varphi_1\rangle=0.$ Then there exist $\rho>0$ such that
	for every $h\in B_{X^\ast}(h^\ast,\rho),$ problem \eqref{1.1} with $\lambda=\lambda_1$ has a solution. If in addition $\langle h,\varphi_1\rangle\ne 0,$ then problem \eqref{1.1} with $\lambda=\lambda_1, $ has two distinct solutions.
\end{theorem}
For the degenerate or linear case $p\geq 2$, we have the following result.


\begin{theorem}\label{V.Theo.degenerate}
	\begin{itemize}
		\item[(i)] Let $p=2$ and $h\in X^\ast$. Then problem~\eqref{1.1} with $\lambda=\lambda_1$ has a solution if and only if $\langle h,\varphi_1\rangle=0.$ When $\langle h,\varphi_1\rangle=0$, there exists a unique function $u^{\bot} \in X^{\bot}$, such that $u\in X$ is a solution of problem~\eqref{1.1} with $\lambda=\lambda_1$ if and only if $u=\tau\varphi_1+u^{\bot}$ for some $\tau\in\mathbb{R}.$ 
		\item[(ii)] Let $2<p<N.$ Let $\textup{(H)}$ and $\textup{(W)}$ be satisfied with $\displaystyle \lim_{\rho\to\infty}\underset{r\geq\rho}{\operatorname{ess\ sup}}\ r^pK(r)=0$. Suppose $h^\ast\in\mathcal{D}_{\varphi_1}^\ast\setminus\{0\}$ be such that $\langle h^\ast,\varphi_1\rangle=0$. Then there exist $\rho>0$ such that problem \eqref{1.1} with $\lambda=\lambda_1,$ $h\in B_{X^\ast}(h^\ast,\rho)$ has a solution. If in addition such an $h$ satisfies $\langle h,\varphi_1\rangle\ne 0,$ then problem \eqref{1.1} with $\lambda=\lambda_1$, has two distinct solutions.
	\end{itemize}	
\end{theorem}


\begin{remark}\rm
	(i) Recall that problem \eqref{1.1} with $\lambda=\lambda_1$ and $h=0$ is the eigenvalue problem \eqref{II.eigen} with $\lambda=\lambda_1$ and hence, all its solutions are of the form $u=\kappa\varphi_1$ $(\kappa\in\mathbb{R}).$

	(ii) Clearly, $u\mapsto \int_{B_1^c}K\varphi_1^{p-1}u\diff x$ is a linear bounded functional on $X$ thanks to the estimate
	$$\left|\int_{B_1^c}K\varphi_1^{p-1}u\diff x\right|\leq \left(\int_{B_1^c}w\varphi_1^p\diff x\right)^{\frac{p-1}{p}}\left(\int_{B_1^c}w|u|^p\diff x\right)^{\frac{1}{p}}\leq C^p_{\operatorname{emb}}\|\varphi_1\|^{p-1}\|u\|,$$
	where $C_{\operatorname{emb}}$ is an embedding constant for  $X\hookrightarrow L^p(B_1^c;w),$ i.e., $\|u\|_{L^p(B_1^c;w)}\leq C_{\operatorname{emb}}\|u\|$ for all $u\in X.$  We identify this functional with $K\varphi_1^{p-1}$ and thus,  $h=h^\ast+\xi K\varphi_1^{p-1}\in  B_{X^{\ast}}(h^\ast,\rho)$ for $|\xi|<\frac{\rho}{C^p_{\operatorname{emb}}\|\varphi_1\|^{p-1}}.$

\end{remark}


\subsection{Auxiliary lemmas}
Since we only deal with the resonant case, 
hereafter we always assume that $\lambda=\lambda_1$ in our arguments. For each $h\in X^\ast$ we denote the energy functional of problem \eqref{1.1} by
$$J_{h}(u):=\frac{1}{p}\int_{B_1^c}|\nabla u|^p\diff x-\frac{\lambda_1}{p}\int_{B_1^c}K(x)|u|^p\diff x-\langle h,u\rangle.$$
This functional is well-defined and belongs to $C^1(X,\mathbb{R})$ with
$$\langle J_h'(u),v\rangle=\int_{B_1^c}|\nabla u|^{p-2}\nabla u\cdot\nabla v\diff x-\lambda_1\int_{B_1^c}K(x)|u|^{p-2}uv\diff x-\langle h,v\rangle,\quad\forall v\in X.$$ 
Clearly, a critical point of $J_{h}$ is a (weak) solution of \eqref{1.1} with $\lambda=\lambda_1.$ We first note that if $\langle h,\varphi_1\rangle\ne 0$ then $J_h$ satisfies the Palais-Smale condition (the  $(PS)$ condition, for short) as shown in the following lemma. 

\begin{lemma}\label{V.le.PS}
	Assume that	$\langle h,\varphi_1\rangle\ne 0$ then $J_h$ satisfies the  $(PS)$ condition for all $1<p<\infty.$
\end{lemma}
\begin{proof}
	The proof is standard. For the reader's convenience we stress it here in our functional setting. Let $c$ be an arbitrary real number. Let $\{u_n\}$ be a $(PS)_c$ sequence in $X$ for $J_h,$ i.e., $J_h(u_n)\to c$ and $J_h'(u_n)\to 0$ as $n\to\infty.$ We first claim that $\{u_n\}$ is bounded in $X$. If this is not the case we may assume that $\|u_n\|\to\infty$ as $n\to\infty.$ Then as $n \to 
	\infty$, we have
	$$-pJ_h(u_n)+\langle J'_h(u_n),u_n\rangle=o(\|u_n\|),$$
	i.e.,
	\begin{equation}\label{V.le1.h}
	(p-1)\langle h,u_n\rangle=o(\|u_n\|).
	\end{equation}
	Set $v_n:=\frac{u_n}{\|u_n\|}$, then up to a subsequence
	\begin{equation}\label{V.le1.lim}
	\begin{cases}
	v_n\rightharpoonup v\ \text{in}\ X,\\
	v_n\to v\ \text{in}\ L^p(B_1^c;w).
	\end{cases}
	\end{equation} 	
	Combining this and \eqref{V.le1.h} we get
	\begin{equation}\label{V.le1.h(v)=0}
	\langle h,v\rangle=0.
	\end{equation}
	Notice that, as $n \to \infty$, we obtain
	\begin{equation}\label{V.le1.eq.vn}
	\frac{1}{p}\int_{B_1^c}|\nabla v_n|^p\diff x-\frac{\lambda_1}{p}\int_{B_1^c}K(x)|v_n|^p\diff x-\frac{1}{\|u_n\|^{p-1}}\langle h,v_n\rangle=\frac{o(\|u_n\|)}{\|u_n\|^p}.
	\end{equation}
	From this and \eqref{V.le1.lim}, we deduce
	$$\frac{1}{p}\int_{B_1^c}|\nabla v|^p\diff x-\frac{\lambda_1}{p}\int_{B_1^c}K(x)|v|^p\diff x\leq 0.$$
	Thus, $v=\kappa\varphi_1$ for some $\kappa\in \mathbb{R}.$ Letting $n\to\infty$ in \eqref{V.le1.eq.vn} and also noticing, $\|v_n\|=1$ for all $n$ we get
	$$1=\lambda_1\int_{B_1^c}K(x)|v|^p\diff x.$$
	Therefore, $\kappa\ne 0$ and then \eqref{V.le1.h(v)=0} gives $\langle h,\varphi_1\rangle=0,$ a contradiction. So $\{u_n\}$ is bounded in $X$. Up to a subsequence we have 
	\begin{equation*}
	\begin{cases}
	u_n\rightharpoonup u\ \text{in}\ X,\\
	u_n\to u\ \text{in}\ L^p(B_1^c;w).
	\end{cases}
	\end{equation*}
	From this, we obtain
	\begin{align*}
	\int_{B_1^c}|&\nabla u_n|^{p-2}\nabla u_n\cdot\nabla(u_n-u)\diff x\\
	&=\langle J_h'(u_n),u_n-u\rangle+\lambda_1\int_{B_1^c}K(x)|u_n|^{p-2}u_n(u_n-u)\diff x+\langle h,u_n-u\rangle\to 0,
	\end{align*}
	as $n\to\infty.$ Since
	$$\int_{B_1^c}|\nabla u_n|^p\diff x\leq p\int_{B_1^c}|\nabla u_n|^{p-2}\nabla u_n\cdot\nabla(u_n-u)\diff x+\int_{B_1^c}|\nabla u|^p\diff x,$$
	we obtain from the last limit and the weak lower semicontinuity of $\|\cdot\|$ on $X$, that
	$$\|u\|\leq \liminf_{n\to\infty}\|u_n\|\leq\limsup_{n\to\infty}\|u_n\|\leq \|u\|.$$
	Thus, $\displaystyle \lim_{n\to\infty}\|u_n\|=\|u\|$. Combining this with the weak convergence of $\{u_n\}$ in $X$, noticing that $X$ is a uniformly convex Banach space, we deduce $u_n\to u$ in $X$ as $n\to\infty.$
\end{proof}
For each $(\tau,h)\in \mathbb{R}\times X^\ast,$ define
$$j(\tau;h):=\inf_{u^{\bot} \in X^{\bot}}J_{h}(\tau\varphi_1+u^{\bot}).$$
To show this infimum is attained at some $u^{\bot}_{\tau,h}\in X^{\bot}$, we need the following lemma.


\begin{lemma}\label{V.le.coercive}
	Let $1<p<\infty$. Then for each $T>0,$ there exist $\alpha_T,\beta_T>0$ such that
	$$	\int_{B_1^c}|\tau\nabla\varphi_1+\nabla u^{\bot}|^p\diff x-\lambda_1	\int_{B_1^c}K(x)|\tau\varphi_1+u^{\bot}|^p\diff x\geq \alpha_T \int_{B_1^c}|\nabla u^{\bot}|^p\diff x-\beta_T$$
	for all $|\tau|\leq T$ and all $u^{\bot}\in X^{\bot}.$
\end{lemma}
\begin{proof}
	Suppose by contradiction that for each $n\in \mathbb{N},$ there are $\tau_n\in [-T,T]$ and $u_n^{\bot}\in X^{\bot}$ such that
	\begin{equation}\label{V.coercive.ineq1}
	\int_{B_1^c}|\tau_n\nabla\varphi_1+\nabla u_n^{\bot}|^p\diff x-\lambda_1	\int_{B_1^c}K(x)|\tau_n\varphi_1+u_n^{\bot}|^p\diff x< \frac{1}{n} \int_{B_1^c}|\nabla u_n^{\bot}|^p\diff x-n.
	\end{equation}
	This yields, $\|u_n^{\bot}\|>n^{\frac{2}{p}}$ for all $n\in\mathbb{N}$ and hence, $\|u_n^{\bot}\|\to\infty$ as $n\to\infty.$ Moreover, \eqref{V.coercive.ineq1} implies
	\begin{equation}\label{V.coercive.ineq2}
	\int_{B_1^c}\left|\frac{\tau_n}{\|u_n^{\bot}\|}\nabla\varphi_1+\nabla \widetilde{u}_n^{\bot}\right|^p\diff x-\lambda_1	\int_{B_1^c}K(x)\left|\frac{\tau_n}{\|u_n^{\bot}\|}\varphi_1+\widetilde{u}_n^{\bot}\right|^p\diff x< \frac{1}{n},
	\end{equation}
	where $	\widetilde{u}_n^{\bot}:=\frac{u^{\bot}_n}{\|u_n^{\bot}\|}\ (n=1,2,\cdots).$ Up to a subsequence, $	\widetilde{u}_n^{\bot}\rightharpoonup \widetilde{u}^{\bot}\in X^{\bot}$ in $X$ as $n\to\infty$ and hence, $\frac{\tau_n}{\|u_n^{\bot}\|}\varphi_1+\widetilde{u}_n^{\bot}\rightharpoonup \widetilde{u}^{\bot}$ in $X$ and $\frac{\tau_n}{\|u_n^{\bot}\|}\varphi_1+\widetilde{u}_n^{\bot}\to \widetilde{u}^{\bot}$ in $L^p(B_1^c;w)$ as $n\to\infty.$ From this, by passing to the limit as $n\to\infty$ in \eqref{V.coercive.ineq2} and recalling the weak lower semicontinuity of norm, we get
	$$\int_{B_1^c}\left|\nabla \widetilde{u}^{\bot}\right|^p\diff x-\lambda_1	\int_{B_1^c}K(x)\left|\widetilde{u}^{\bot}\right|^p\diff x\leq 0.$$
	Thus $\widetilde{u}^{\bot}=\kappa \varphi_1$ for some $\kappa\in\mathbb{R}$ and hence, $\widetilde{u}^{\bot}=0$ since $\widetilde{u}^{\bot}\in X^{\bot}.$  Meanwhile, we have
	\begin{align*}
	\int_{B_1^c}\left|\frac{\tau_n}{\|u_n^{\bot}\|}\nabla\varphi_1+\nabla \widetilde{u}_n^{\bot}\right|^p\diff x&\geq \frac{1}{2^{p-1}}\int_{B_1^c}\left|\nabla \widetilde{u}_n^{\bot}\right|^p\diff x-\frac{|\tau_n|^p}{\|u_n^{\bot}\|^p}\int_{B_1^c}|\nabla\varphi_1|^p\diff x\\&=\frac{1}{2^{p-1}}-\frac{|\tau_n|^p}{\|u_n^{\bot}\|^p}\lambda_1.
	\end{align*}
	Combining this with the facts that
	$\frac{\tau_n}{\|u_n^{\bot}\|}\to 0$ in $\mathbb{R}$, $\frac{\tau_n}{\|u_n^{\bot}\|}\varphi_1+\widetilde{u}_n^{\bot}\to \widetilde{u}^{\bot}$ in $L^p(B_1^c;w)$ as $n\to\infty$, and using \eqref{V.coercive.ineq2}, we obtain 
	$\int_{B_1^c}K(x)\left|\widetilde{u}^{\bot}\right|^p\diff x>0,$
	a contradiction. So we have just proved Lemma~\ref{V.le.coercive}.	
\end{proof}
\begin{remark}\rm
	
We point out that we can prove Lemma~\ref{V.le.coercive} also using the fact that for $0<\gamma\leq\infty,$ we have 
\begin{equation*}\label{IV.Lambda_gamma}
\Lambda_\gamma:=\inf\left\{\int_{B_1^c}|\nabla u|^p\diff x: u\in \mathcal{C}'_\gamma\setminus\{0\}, \int_{B_1^c} K(x)|u|^p\diff x=1 \right\}>\lambda_1, \end{equation*}
where $\mathcal{C}'_\gamma:=\{u=\tau\varphi_1+u^{\bot}: \tau\in\mathbb{R}, u^{\bot}\in X^{\bot}, \|u^{\bot}\|\geq \gamma|\tau|\}$ when $\gamma\in (0,\infty)$ and $\mathcal{C}'_\infty:=X^{\bot}$ (see \cite[Lemma 6.2 and Subsection 8.2]{Alziary}). However, here we provided a direct proof without using the argument on cones as in \cite{Alziary}. 
\end{remark}

 By Lemma~\ref{V.le.coercive}, it is easy to see that for each $(\tau,h)\in\mathbb{R}\times X^\ast$, the functional $u^{\bot}\mapsto J_h(\tau\varphi_1+u^{\bot})$ is coercive on $X^{\bot}$. Moreover, this functional is weak lower semicontinuous on $X^{\bot}$, so it achieves a global minimum on $X^{\bot}$ at some $u^{\bot}_{\tau,h}\in X^{\bot},$ that is
\begin{equation}\label{V.form_j}
J_h(\tau\varphi_1+u^{\bot}_{\tau,h})=\inf_{u^{\bot}\in X^{\bot}}J_h(\tau\varphi_1+u^{\bot})=j(\tau;h).
\end{equation}
When $h=h^\ast$ with a fixed $h^\ast\in X^\ast$ satisfying $\langle h^\ast,\varphi_1\rangle=0$, we write $u^{\bot}_\tau$ instead of $u^{\bot}_{\tau,h^\ast}.$ If $1<p<2$ and 
$K\in L^\infty(A_1^R)$ for all $R>1,$ then by Proposition~\ref{III.Prop.saddle}, we get
\begin{equation}\label{V.lim_j}
\lim_{|\tau|\to\infty}j(\tau;h^\ast)=-\infty.
\end{equation}
In the next lemma, we stress a behavior of $u^{\bot}_{\tau,h}$ and $j(\tau;h^\ast)$ as $|\tau|\to\infty.$


\begin{lemma}\label{V.le.behavior}
	Let $1<p<2$ and let 
	$K\in L^\infty(A_1^R)$ for all $R>1.$ Then for every $h\in X^\ast,$ we have
	$$\limsup_{|\tau|\to\infty}\|\frac{u^{\bot}_{\tau,h}}{\tau}\|<\infty\ \text{and}\ \lim_{|\tau|\to\infty}\frac{j(\tau;h^\ast)}{\tau}=0.$$
\end{lemma}
\begin{proof}
	We first show that
	\begin{equation}\label{V.limsup1}
	\limsup_{|\tau|\to\infty}\|\frac{u_\tau^{\bot}}{\tau}\|<\infty.
	\end{equation} 
	If \eqref{V.limsup1} does not hold true, then we can find a sequence $\{\tau_n\}$ such that $|\tau_n|\to\infty$ and $\|\frac{u_{\tau_n}^{\bot}}{\tau_n}\|\to\infty$ as $n\to\infty,$ i.e.,  $\frac{|\tau_n|}{\|u_{\tau_n}^{\bot}\|}\to 0$ as $n\to\infty.$ Set $v_n:=\frac{u_{\tau_n}^{\bot}}{\|u_{\tau_n}^{\bot}\|}$ then up to a subsequence, we have
	\begin{equation*}
	\begin{cases}
	v_n\rightharpoonup v_0\ \text{in}\ X,\\
	v_n\to v_0\ \text{in}\ L^p(B_1^c;w),\\
	v_0\in X^{\bot}.
	\end{cases}
	\end{equation*}
	From this and \eqref{V.lim_j}, we deduce 
	\begin{align}\label{V1.est.int1}
	\frac{1}{p}\int_{B_1^c}|\nabla v_0|^p\diff x-&\frac{\lambda_1}{p}\int_{B_1^c}K(x)|v_0|^p\diff x\leq\underset{n\to\infty}{\lim\inf}\left\{\frac{1}{p}\int_{B_1^c}\left|\frac{\tau_n}{\|u_{\tau_n}^{\bot}\|}\nabla\varphi_1+\nabla v_n\right|^p\diff x\right.\notag\\
	&\left.-\frac{\lambda_1}{p}\int_{B_1^c}K(x)\left|\frac{\tau_n}{\|u_{\tau_n}^{\bot}\|}\varphi_1+v_n\right|^p\diff x-\frac{1}{\|u_{\tau_n}^{\bot}\|^{p-1}}\langle h^\ast,v_n\rangle\right\}\leq 0.
	\end{align}
	Thus, $v_0=\kappa\varphi_1$ for some $\kappa\in\mathbb{R}$ and hence $v_0=0$ since $v_0\in X^{\bot}.$ Meanwhile, arguing as in the proof of Lemma~\ref{V.le.coercive}, we obtain from \eqref{V1.est.int1} that
	$\int_{B_1^c}K(x)|v_0|^{p}\diff x>0,$ which is absurd. 
	Thus \eqref{V.limsup1} holds true. Next suppose that for some $h\in X^\ast,$ there exists a sequence $\{\tau_n'\}$ such that $|\tau_n'|\to\infty$ and $\|\frac{u_{\tau'_n,h}^{\bot}}{\tau'_n}\|\to\infty$ as $n\to\infty.$   Since
	$$J_h(\tau_n'\varphi_1+u_{\tau'_n,h}^{\bot}) \leq J_h(\tau_n'\varphi_1+u_{\tau'_n}^{\bot})=j(\tau_n';h^\ast)-\langle h-h^\ast,\tau_n'\varphi_1+u_{\tau'_n}^{\bot}\rangle, $$
	we deduce
	\begin{align}\label{V1.est.int2}
	&\frac{1}{p}\int_{B_1^c}\left|\frac{\tau'_n}{\|u_{\tau'_n,h}^{\bot}\|}\nabla\varphi_1+\nabla \widetilde{v}_n\right|^p\diff x-\frac{\lambda_1}{p}\int_{B_1^c}K(x)\left|\frac{\tau'_n}{\|u_{\tau'_n,h}^{\bot}\|}\varphi_1+\widetilde{v}_n\right|^p\diff x\notag\\
	&\quad-\frac{1}{\|u_{\tau'_n,h}^{\bot}\|^{p-1}}\left\langle h,\frac{\tau'_n}{\|u_{\tau'_n,h}^{\bot}\|}\varphi_1+\widetilde{v}_n\right\rangle+\frac{1}{\|u_{\tau'_n,h}^{\bot}\|^{p-1}}\left\langle h- h^\ast,\frac{\tau'_n}{\|u_{\tau'_n,h}^{\bot}\|}\varphi_1+\frac{\tau'_n}{\|u_{\tau'_n,h}^{\bot}\|}\frac{u_{\tau'_n}^{\bot}}{\tau'_n}\right\rangle\notag\\
	&\leq \frac{j(\tau_n';h^\ast)}{\|u_{\tau'_n,h}^{\bot}\|^p},\ \text{where}\ \widetilde{v}_n:=\frac{u^\bot_{\tau'_n,h}}{\|u^\bot_{\tau'_n,h}\|}.
	\end{align}
	Up to a subsequence, we have
	\begin{equation}\label{V.convergence1}
	\begin{cases}
	\widetilde{v}_n\rightharpoonup \widetilde{v}_0\ \text{in}\ X,\\
	\widetilde{v}_n\to \widetilde{v}_0\ \text{in}\ L^p(B_1^c;w),\\
	\widetilde{v}_0\in X^{\bot}.
	\end{cases}
	\end{equation}
	Letting $n\to\infty$ in \eqref{V1.est.int2},  and using \eqref{V.lim_j}, \eqref{V.limsup1}, \eqref {V.convergence1} and $\frac{|\tau_n'|}{\|u_{\tau_n',h}^{\bot}\|}\to 0$ as $n\to\infty$, we obtain
	$$\frac{1}{p}\int_{B_1^c}|\nabla \widetilde{v}_0|^p\diff x-\frac{\lambda_1}{p}\int_{B_1^c}K(x)|\widetilde{v}_0|^p\diff x\leq 0.$$
	Thus, $\widetilde{v}_0=\kappa'\varphi_1$ for some $\kappa'\in\mathbb{R}$, and hence $\widetilde{v}_0=0$ since $\widetilde{v}_0\in X^{\bot}.$ Arguing again as in the proof of Lemma~\ref{V.le.coercive}, we obtain from \eqref{V1.est.int2} that
	$\int_{B_1^c}K(x)|\widetilde{v}_0|^p\diff x>0,$ a contradiction. So we get
	\begin{equation}\label{V.limsup2}
	\limsup_{|\tau|\to\infty}\|\frac{u^{\bot}_{\tau,h}}{\tau}\|<\infty.
	\end{equation}
	
	Next, we show that
	\begin{equation}\label{V.lim_h}
	\exists\lim_{|\tau|\to\infty}\left\langle h,\frac{u_{\tau,h}^{\bot}}{|\tau|}\right\rangle=0.
	\end{equation}
	Suppose by contradiction that there exists a sequence $\{\widetilde{\tau}_n\}$ such that $|\widetilde{\tau}_n|\to \infty$ as $n\to\infty$ and $\underset{n\to\infty}{\lim\inf}|\langle h,\frac{u_{\widetilde{\tau}_n,h}^{\bot}}{\widetilde{\tau}_n}\rangle|>0$. Using \eqref{V.limsup2}, we deduce (up to a subsequence)
	\begin{equation}\label{V.convergence2}
	\frac{u_{\widetilde{\tau}_n,h}^{\bot}}{\widetilde{\tau}_n}\rightharpoonup u_0\ \ \text{in}\ \ X,\ \ \frac{u_{\widetilde{\tau}_n,h}^{\bot}}{\widetilde{\tau}_n}\to u_0\ \ \text{in}\ \ L^p(B_1^c;w),\ u_0\in X^{\bot},
	\end{equation}
	and
	\begin{equation}\label{V.h}
	0<\underset{n\to\infty}{\lim\inf}\left|\left\langle h,\frac{u_{\widetilde{\tau}_n,h}^{\bot}}{\widetilde{\tau}_n}\right\rangle\right|=|\langle h,u_0\rangle|.
	\end{equation}
	Since
	$$J_h(\widetilde{\tau}_n\varphi_1+u_{\widetilde{\tau}_n,h}^{\bot})+\langle h-h^\ast,\widetilde{\tau}_n\varphi_1+u_{\widetilde{\tau}_n}^{\bot}\rangle\leq j(\widetilde{\tau}_n;h^\ast), $$
	we deduce 
	\begin{align}\label{V.est.int3}
	\frac{1}{p}\int_{B_1^c}\left|\nabla\varphi_1+\nabla\left(\frac{u_{\widetilde{\tau}_n,h}^{\bot}}{\widetilde{\tau}_n}\right)\right|^p\diff x&-\frac{\lambda_1}{p}\int_{B_1^c}K(x)\left|\varphi_1+\frac{u_{\widetilde{\tau}_n,h}^{\bot}}{\widetilde{\tau}_n}\right|^p\diff x-\frac{\widetilde{\tau}_n}{|\widetilde{\tau}_n|^p}\left\langle h,\varphi_1+\frac{u_{\widetilde{\tau}_n,h}^{\bot}}{\widetilde{\tau}_n}\right\rangle\notag\\
	&+\frac{\widetilde{\tau}_n}{|\widetilde{\tau}_n|^p}\left\langle h-h^\ast,\varphi_1+\frac{u_{\widetilde{\tau}_n}^{\bot}}{\widetilde{\tau}_n}\right\rangle\leq \frac{j(\widetilde{\tau}_n;h^\ast)}{|\widetilde{\tau}_n|^p}.
	\end{align}
	Now, taking the limit $n\to\infty$ in \eqref{V.est.int3} and invoking \eqref{V.lim_j}, \eqref{V.limsup2} and \eqref{V.convergence2}, we get 
	$$\frac{1}{p}\int_{B_1^c}|\nabla\varphi_1+\nabla u_0|^p\diff x-\frac{\lambda_1}{p}\int_{B_1^c}K(x)|\varphi_1+u_0|^p\diff x\leq 0,$$
	and hence, $\varphi_1+u_0=\widetilde{\kappa}\varphi_1$  for some $\widetilde{\kappa}\in\mathbb{R}.$ Thus, $u_0=0$ due to the fact that $u_0\in X^{\bot}.$ This contradicts to \eqref{V.h} and hence, we obtain \eqref{V.lim_h}. Finally, the second conclusion of lemma follows from \eqref{V.lim_h} and the following estimate
	\begin{align*}
	-\left\langle h^\ast,\frac{u_{\tau}^{\bot}}{|\tau|}\right\rangle\leq\frac{1}{|\tau|}\biggl\{\frac{1}{p}\int_{B_1^c}|\tau\nabla\varphi_1&+\nabla u_\tau^{\bot}|^p\diff x-\frac{\lambda_1}{p}\int_{B_1^c}K(x)|\tau\varphi_1+u_\tau^{\bot}|^p\diff x\notag\\
	&-\langle h^\ast,\tau\varphi_1+u_\tau^{\bot}\rangle\biggr\}=\frac{j(\tau;h^\ast)}{|\tau|}\leq 0,\quad \text{for}\ |\tau|\ \text{large.}
	\end{align*}\end{proof}


\begin{remark}\rm\label{V.rmk.behavior}
	Let $1<p<2.$ From the arguments in the proof of Lemma~\ref{V.le.behavior}, it is easy to see that 
	for each $h\in X^\ast$ there are two sequences $\{\tau_n\}$ and $\{\tau'_n\}$ in $\mathbb{R}$, such that $\tau_n\to\infty$ and $\tau'_n\to-\infty$ as $n\to\infty$ and
	$$\frac{u_{\tau_n,h}^{\bot}}{\tau_n}\rightharpoonup 0, \frac{u_{\tau'_n,h}^{\bot}}{\tau'_n}\rightharpoonup 0\quad \text{in}\ X\ \text{as}\ n\to\infty.$$
\end{remark}

The next lemma provides the continuity of $j(\cdot;\cdot)$ on $\mathbb{R}\times X^\ast.$


\begin{lemma}\label{V.le.continuity}
	Let $1<p<\infty$. Then, $j(\cdot;\cdot): \mathbb{R}\times X^\ast \to \mathbb{R}$ is a continuous mapping.
\end{lemma}
\begin{proof}
	First, we claim that for $|\tau|\leq T_0$ and $\|h\|_{X^\ast}\leq M_0$ we have
	\begin{equation}\label{V.est.u^T}
\|u^{\bot}_{\tau,h}\|\leq \left[\frac{p\beta_{T_0}}{(p-1)\alpha_{T_0}}+\alpha_{T_0}^{-\frac{p}{p-1}}M_0^{\frac{p}{p-1}}\right]^{\frac{1}{p}},
	\end{equation}
where $\alpha_{T_0},\beta_{T_0}$ depend only on $T_0$ as in Lemma~\ref{V.le.coercive}. Indeed, by Lemma~\ref{V.le.coercive} and Young inequality, for all $|\tau|\leq T_0$ we have
\begin{align*}
\alpha_{T_0}\|u_{\tau,h}^{\bot}\|^p-\beta_{T_0}&\leq\frac{1}{p}\int_{B_1^c}|\nabla(\tau\varphi_1+u_{\tau,h}^{\bot})|^p\diff x-\frac{\lambda_1}{p}\int_{B_1^c}K(x)|\tau\varphi_1+u_{\tau,h}^{\bot}|^p\diff x\notag\\
&= J_{h}(\tau\varphi_1+u_{\tau,h}^{\bot})+\langle h,\tau\varphi_1+u_{\tau,h}^{\bot}\rangle\\
&\leq J_{h}(\tau\varphi_1)+\tau\langle h,\varphi_1\rangle+\langle h,u_{\tau,h}^{\bot}\rangle=\langle h,u_{\tau,h}^{\bot}\rangle\\
&\leq\|h\|_{X^\ast}\|u_{\tau,h}^{\bot}\|\leq \frac{\alpha_{T_0}}{p}\|u_{\tau,h}^{\bot}\|^p+\frac{p-1}{p}\alpha_{T_0}^{-\frac{1}{p-1}}\|h\|_{X^\ast}^{\frac{p}{p-1}}.
\end{align*}
Thus, we obtain \eqref{V.est.u^T}. Now, let $(\tau_n,h_n)\to (\tau_0,h_0)$  in $\mathbb{R}\times X^\ast$ as $n\to\infty.$ Let $\{u_n^{\bot}\}\subset X^{\bot}$ be such that, $j(\tau_n;h_n)=J_{h_n}(\tau_n\varphi_1+u_n^{\bot})$ for all $n.$ By \eqref{V.est.u^T}, $\{u_n^{\bot}\}$ is a bounded sequence in $X.$ So, up to a subsequence, $u_n^{\bot}\rightharpoonup w^{\bot}$ in $X$ and hence, $w^{\bot}\in X^{\bot},$  $\tau_n\varphi_1+u_n^{\bot}\rightharpoonup \tau_0\varphi_1+w^{\bot}$ in $X$ and $\tau_n\varphi_1+u_n^{\bot}\to \tau_0\varphi_1+w^{\bot}$ in $L^p(B_1^c;w)$ as $n\to\infty.$ Thus
	\begin{align*}
	\underset{n\to\infty}{\lim\inf}j(\tau_n;h_n)&=\underset{n\to\infty}{\lim\inf}J_{h_n}(\tau_n\varphi_1+u_n^{\bot})\\
	&=\underset{n\to\infty}{\lim\inf}\biggl[\frac{1}{p}\int_{B_1^c}|\tau_n\nabla\varphi_1+\nabla u_n^{\bot} |^p\diff x-\frac{\lambda_1}{p}\int_{B_1^c}K(x)|\tau_n\varphi_1+ u_n^{\bot}|^p\diff x\\&\qquad \qquad-\langle h_n,\tau_n\varphi_1+ u_n^{\bot}\rangle\biggr]\\
	&\geq \frac{1}{p}\int_{B_1^c}|\tau_0\nabla\varphi_1+\nabla w^{\bot} |^p\diff x-\frac{\lambda_1}{p}\int_{B_1^c}K(x)|\tau_0\varphi_1+ w^{\bot}|^p\diff x\\&\qquad \qquad-\langle h_0,\tau_0\varphi_1+ w^{\bot}\rangle\\
	&=J_{h_0}(\tau_0\varphi_1+w^{\bot})\geq j(\tau_0;h_0).
	\end{align*}
	On the other hand, if $u_0^{\bot}$ is a global minimizer for the functional $u^{\bot}\mapsto 
	J_{h_0}(\tau_0\varphi_1+u^{\bot})$ on $X^{\bot}$, then
	\begin{align*}
	\underset{n\to\infty}{\lim\sup}\ j(\tau_n;h_n)&=\underset{n\to\infty}{\lim\sup}J_{h_n}(\tau_n\varphi_1+u_n^{\bot})\\&\leq \lim_{n\to\infty}J_{h_n}(\tau_n\varphi_1+u_0^{\bot})=J_{h_0}(\tau_0\varphi_1+u_0^{\bot})= j(\tau_0;h_0).
	\end{align*}
	Thus, we obtain
	$$J_{h_0}(\tau_0\varphi_1+w^{\bot})=j(\tau_0;h_0)=\lim_{n\to\infty}j(\tau_n;h_n),$$
	and this proves Lemma~\ref{V.le.continuity}.
\end{proof}
As shown in \cite[Remark 8.2]{Alziary}, $w^{\bot}$ above is indeed a global minimizer for $u^{\bot}\mapsto J_{h_0}(\tau_0\varphi_1+u^{\bot})$ and $u_n^{\bot}\to w^{\bot}$ in $X$ as $n\to\infty.$ Proof of the next lemma can be found in \cite{Alziary}. Indeed, a careful inspection of the proof of Lemma 8.3 in \cite{Alziary} shows that it remains valid even when $\mathcal{D}^{1,p}(\mathbb{R}^N)$ is replaced by $X$.


\begin{lemma}\label{V.le.maximum}
	Let $1<p<\infty$ and let $h\in X^\ast$ be given. Assume that $j(\cdot;h):\mathbb{R}\to\mathbb{R}$ attains a local maximum $m_0$ at some $\tau_0\in\mathbb{R}.$ Then there exists $u_0^{\bot}\in X^{\bot}$ such that $u_0^{\bot}$ is a global minimizer for the functional $u^{\bot}\mapsto J_h(\tau_0\varphi_1+u^{\bot})$ on $X^{\bot},$ $u_0=\tau_0\varphi_1+u_0^{\bot}$ is a critical point for $J_{h}$ and $J_{h}(u_0)=m_0.$
\end{lemma}

Finally, we need the following auxiliary result.


\begin{lemma}\label{V.le.nonnegativeness}
	Let $1<p<\infty$ and $h\in X^\ast.$ For $M>0,C>0$ given, there exists $R>C$ such that for all $\tau\in[-M,M]$ and all $u^{\bot}\in X^{\bot},$ $\|u^{\bot}\|=R,$ we have
	$$J_h(\tau\varphi_1+u^{\bot})\geq 0.$$
\end{lemma}

\begin{proof}
	If the conclusion is not true, then for each $n\in\mathbb{N}$ there exist $\tau_n\in [-M,M],$  and $u_n^{\bot}\in X^{\bot}$ with $\|u_n^{\bot}\|=\max\{n,C+1\}$, such that
	\begin{align*}
	J_h(\tau_n\varphi_1+u_n^{\bot})&=\frac{1}{p}\int_{B_1^c}|\tau_n\nabla\varphi_1+\nabla u_n^{\bot}|^p\diff x-\frac{\lambda_1}{p}\int_{B_1^c}K(x)|\tau_n\varphi_1+u_n^{\bot}|^p\diff x\\&\quad-\langle h,\tau_n\varphi_1+u_n^{\bot}\rangle<0.
	\end{align*}
	Thus,
	\begin{align}\label{V.le.nonnegativeness.int}
	\frac{1}{p}\int_{B_1^c}\left|\frac{\tau_n}{\|u_n^{\bot}\|}\nabla\varphi_1+\nabla w_n\right|^p\diff x-\frac{\lambda_1}{p}\int_{B_1^c}&K(x)\left|\frac{\tau_n}{\|u_n^{\bot}\|}\varphi_1+w_n\right|^p\diff x\notag\\
	&-\frac{1}{\|u_n^{\bot}\|^{p-1}}\left\langle h,\frac{\tau_n}{\|u_n^{\bot}\|}\varphi_1+w_n\right\rangle<0,
	\end{align}
	where $w_n:=\frac{u_n^{\bot}}{\|u_n^{\bot}\|}$ for all $n\in\mathbb{N}.$ We may assume
	\begin{equation}\label{V.le.nonnegativeness.conv}
	\begin{cases}
	w_n\rightharpoonup w_0\ \text{in}\ X,\\
	w_n\to w_0\ \text{in}\ L^p(B_1^c;w),\\
	w_0\in X^{\bot}.
	\end{cases}
	\end{equation}
	Letting $n\to\infty$ in \eqref{V.le.nonnegativeness.int} and invoking \eqref{V.le.nonnegativeness.conv}, we obtain
	
	$$\frac{1}{p}\int_{B_1^c}|\nabla w_0|^p\diff x-\frac{\lambda_1}{p}\int_{B_1^c}K(x)|w_0|^p\diff x\leq 0,$$
	and hence $w_0=0.$ But combining \eqref{V.le.nonnegativeness.int} with the facts that $\frac{\tau_n}{\|u_n^{\bot}\|}\to 0$ as $n\to\infty$ and $\|w_n\|=1$ for all $n$, we argue as in the proof of Lemma~\ref{V.le.coercive} to conclude
	$$\int_{B_1^c}K(x)|w_0|^p\diff x>0,$$ a contradiction. The proof is complete.
\end{proof}


\subsection{Proofs of Theorems \ref{V.Theo.singular} and \ref{V.Theo.degenerate}}

\begin{proof}[Proof of Theorem~\ref{V.Theo.singular}]
	For each $\tau\in\mathbb{R}$ and $h\in X^\ast$, define $u^{\bot}_\tau, u_{\tau,h}^{\bot}$ as in \eqref{V.form_j}. Using \eqref{V.lim_j}, we can find $M_1>0>M_2$, such that
	\begin{equation}\label{V.Theo.sing.0}
\max\{j(M_1;h^\ast),j(M_2;h^\ast)\}<3j(0;h^\ast)<j(0;h^\ast).
	\end{equation}
	Here we note that $j(0;h^\ast)<0.$ Applying \eqref{V.est.u^T} for $T_0=0$ and $M_0=\|h^\ast\|_{X^\ast}+1$ we have
	\begin{equation}\label{V.Theo.sing.u0,h}
	\|u_{0,h}^{\bot}\|\leq C_{h^\ast},\ \forall h\in B_{X^\ast}(h^\ast,1),
	\end{equation} 
	with $C_{h^\ast}>0$ depending only on $h^\ast.$ Let
	\begin{equation}\label{V.Theo.sing.rhotilde}
	0<\widetilde{\rho}<\min\left\{1,-\frac{j(0;h^\ast)}{C_{h^\ast}},-\frac{j(0;h^\ast)}{\|M_1\varphi_1+u_{M_1}^{\bot}\|},-\frac{j(0;h^\ast)}{\|M_2\varphi_1+u_{M_2}^{\bot}\|}\right\}.
	\end{equation}
Recall that $j$ is continuous on $\mathbb{R}\times X^\ast$ in view of Lemma~\ref{V.le.continuity}. Thus, by \eqref{V.Theo.sing.0} there exists $\rho\in (0,\widetilde{\rho})$ such that
\begin{equation}\label{V.Theo.sing.max}
	\max\left\{j\left(M_1;h\right),j\left(M_2;h\right)\right\}<j(0;h)
\end{equation}
for all $h\in B_{X^\ast}(h^\ast,\rho).$ Let $h\in B_{X^\ast}(h^\ast,\rho)$ and consider the following cases.
	
	\textbf{The case $\mathbf{\langle h,\varphi_1\rangle\ne 0.}$} We only treat the case $\langle h,\varphi_1\rangle<0$, since the other case $\langle h,\varphi_1\rangle> 0$ can be treated similarly.  Using \eqref{V.Theo.sing.0}-\eqref{V.Theo.sing.rhotilde}, we estimate
	\begin{align}\label{V.Theo.sing.1}
	J_h(M_1\varphi_1+u_{M_1}^{\bot})&=j(M_1;h^\ast)-\langle h-h^\ast,M_1\varphi_1+u_{M_1}^{\bot}\rangle\notag\\
	&\leq 3j(0;h^\ast)+\rho\|M_1\varphi_1+u_{M_1}^{\bot}\|\notag\\
	&<2j(0;h^\ast)\leq J_h(u^{\bot}_{0,h})+\langle h-h^\ast,u_{0,h}^{\bot}\rangle+ j(0;h^\ast)\notag\\
	&<J_h(u^{\bot}_{0,h})\leq 0.
	\end{align}
	By Remark~\ref{V.rmk.behavior}, we find a sequence $\{\tau_n\}\subset \mathbb{R}$, such that $\tau_n\to\infty$ and $\frac{u^{\bot}_{\tau_n,h}}{\tau_n}\rightharpoonup 0$ in $X$ as $n\to\infty.$ Thus
	\begin{align*}
	J_h(\tau_n \varphi_1+u_{\tau_n,h}^{\bot})&=J_{h^\ast}(\tau_n\varphi_1+u_{\tau_n,h}^{\bot})-\langle h-h^\ast,\tau_n\varphi_1+u_{\tau_n,h}^{\bot}\rangle\\
	&\geq j(\tau_n;h^\ast)-\langle h-h^\ast,\tau_n\varphi_1+u_{\tau_n,h}^{\bot}\rangle\\&=\tau_n\left(\frac{j(\tau_n;h^\ast)}{\tau_n}-\langle h-h^\ast,\frac{u^{\bot}_{\tau_n,h}}{\tau_n}\rangle-\langle h,\varphi_1\rangle\right).
	\end{align*}
	From this, $\frac{u^{\bot}_{\tau_n,h}}{\tau_n}\rightharpoonup 0$ and Lemma~\ref{V.le.behavior}, we deduce
	$$J_h(\tau_n\varphi_1+u_{\tau_n,h}^{\bot})>0,\ \ \text{for}\ n\ \text{large.}$$
	Hence, there exists $M>M_1$ such that
	\begin{equation}\label{V.Theo.sing.2}
	J_h(M\varphi_1+u_{M,h}^{\bot})>0.
	\end{equation}
	Then, by Lemma~\ref{V.le.nonnegativeness} there is $R> \|u^{\bot}_{M_1}\|$ 
	such that
	\begin{equation}\label{V.Theo.sing.3}
	J_h(\tau\varphi_1+u^{\bot})\geq 0 \ \ \text{for all}\ \tau\in\mathbb{R},\ u^{\bot}\in X^{\bot}\ \text{with}\ 0\leq \tau\leq M \ \text{and}\  \|u^{\bot}\|=R.
	\end{equation}
	Set
	$$D:=\{u\in X: u=\tau\varphi_1+u^{\bot},\tau\in[0,M], u^{\bot}\in X^\bot,\|u^{\bot}\|\leq R\}.$$
	Then $D$ is bounded and weakly closed subset of $X$ with the boundary $\partial D := \partial_1D \cup \partial_2 D \cup \partial_3 D$, where 
	\begin{align*}
         \partial_1 D &= \{u \in X: u = u^{\bot},u^{\bot}\in X^\bot, \|u^{\bot}\| \leq R\},\\
         \partial_2 D &= \{u \in X: u = M\varphi_1 + u^{\bot},u^{\bot}\in X^\bot, \|u^{\bot}\|\leq R\}, \\
         \partial_3 D &=\{u \in X: u = \tau \varphi_1+u^{\bot}, \tau \in [0,M],u^{\bot}\in X^\bot, \|u^{\bot}\|=R\}.
	\end{align*}
	It follows from \eqref{V.Theo.sing.1} that 
	\begin{align}
	\inf_{u \in D} J_h(u) \leq J_h(M_1\varphi_1+u^{\bot}_{M_1}) < J_h(u^{\bot}_{0,h}) = \inf_{u^{\bot} \in X^{\bot}} J_h(u^{\bot}) \leq \inf_{u \in \partial_1 D}J_h(u).    \label{V.Theo.sing.B1}
	\end{align}
	It follows from \eqref{V.Theo.sing.1} and \eqref{V.Theo.sing.2} that
	\begin{align}
	 \inf_{u \in D} J_h(u) &\leq J_h(M_1\varphi_1+u^{\bot}_{M_1}) < 0 < J_h(M \varphi_1+u^{\bot}_{M,h}) = \inf_{u^{\bot} \in X^{\bot}}J_h(M\varphi_1+u^{\bot}) \notag \\ &\leq \inf_{u \in \partial_2 D} J_h(u). \label{V.Theo.sing.B2}
	 \end{align}
	 Finally, it follows from \eqref{V.Theo.sing.1} and \eqref{V.Theo.sing.3} that  
	 \begin{align}
	 \inf_{u \in D} J_h(u) <0 \leq \inf_{\partial_3 D}J_h(u).   \label{V.Theo.sing.B3}
	 \end{align}
    Therefore, from \eqref{V.Theo.sing.B1}-\eqref{V.Theo.sing.B3}, we conclude 	 
	\begin{equation}\label{V.local.min.geo}
	\inf_{u\in D}J_h(u)<\inf_{u\in\partial D}J_h(u).
	\end{equation} By the weak lower semicontinuity of $J_{h}$ on $X$, $\displaystyle\inf_{u\in D}J_h(u)$ is attained at some $u_D\in D.$ By \eqref{V.local.min.geo}, $u_D$ is an interior point of $D$ and thus $u_D$ is a critical point of $J_h,$ i.e., a solution of \eqref{1.1}. Since $J_{h}(t\varphi_1)\to -\infty$ as $t\to -\infty,$ $J_{h}$ is unbounded from below. Thus, $J_{h}$ has a Mountain Pass geometry and hence, in view of Lemma~\ref{V.le.PS}, we can apply Mountain Pass Theorem to obtain a Mountain Pass solution $u_0\in X$ of \eqref{1.1} such that $u_0\ne u_D$ and is also a critical point for $J_{h}.$
	
	\textbf{The case $\mathbf{\langle h,\varphi_1\rangle=0.}$} From \eqref{V.Theo.sing.max}, the continuous function $j(\cdot;h):\ \mathbb{R}\to \mathbb{R}$ attains a local maximum at some $\tau_0\in (M_2,M_1)$. Then by Lemma~\ref{V.le.maximum}, for some global minimizer $u_{\tau_0,h}^{\bot}$ of the functional $u^{\bot}\mapsto J_h(\tau_0\varphi_1+u^{\bot})$ on $X^{\bot}$ we have that $\widetilde{u}_0=\tau_0\varphi_1+u_{\tau_0,h}^{\bot}$ is a critical point of $J_h$ and hence, a solution of problem \eqref{1.1}.
\end{proof}


\begin{proof}[Proof of Theorem~\ref{V.Theo.degenerate}]
	
	(i) The proof of this  part can be obtained easily by applying any well known technique (e.g. Lax-Milgram theorem, direct
	methods of the Calculus of Variation) for linear elliptic equations in which we invoke \eqref{IV.Poincarep=2} or the compact embedding $X\hookrightarrow\hookrightarrow L^2(B_1^c; w)$.
	
	(ii) Since $h^\ast\in\mathcal{D}_{\varphi_1}^\ast\setminus\{0\},$ the density of $C_c^1(B_1^c)$ in $\mathcal{D}_{\varphi_1}$ implies that there exists $\phi\in C_c^1(B_1^c)$ such that $\langle h^\ast,\phi\rangle>0.$ Then for $t>0$ small,
	\begin{equation}\label{V2.negativeness.of.J}
	J_{h^\ast}(t\phi)=t^p\left[\frac{1}{p}\int_{B_1^c}|\nabla\phi|^p\diff x-\frac{\lambda_1}{p}\int_{B_1^c}K(x)|\phi|^p\diff x\right]-t\langle h^\ast,\phi\rangle <0.
	\end{equation}
	Repeating the argument used in \cite[Proof of Theorem 1.2]{Drabek.EJDE2002}, using the improved Poincar\'e inequality \eqref{IV.Poincare}, the embedding $X\hookrightarrow\mathcal{D}_{\varphi_1}$ and  \eqref{V2.negativeness.of.J}, we can find $R>0$ and $T>0$ such that	
	\begin{equation*}\label{V2.est.infJ}
	\inf_{u\in D}J_{h^\ast}(u)<\inf_{u\in \partial D}J_{h^\ast}(u),
	\end{equation*}
	where $D:=\{u\in X:\ u=\tau\varphi_1+u^\bot,\tau\in [-T,T], u^{\bot}\in X^\bot, \|u^\bot\|\leq R \}.$
	
Then, let 
	\begin{equation}\label{V.choose.rho}
	0<\rho<\frac{\inf_{u\in \partial D}J_{h^\ast}(u)-\inf_{u\in  D}J_{h^\ast}(u)}{2M},
	\end{equation}
	where $M:=\sup_{u\in D}\|u\|.$  Let $h\in B_{X^\ast}(h^\ast,\rho).$ 
	Let $\{u_n\}\subset D$ be such that $J_{h^\ast}(u_n)\to\inf_{u\in D}J_{h^\ast}(u)$ as $n\to\infty.$ Then we have
	\begin{equation}\label{V.inf.D}
	\inf_{u\in D}J_{h^\ast}(u)=\lim_{n\to\infty}\left[J_{h}(u_n)+\langle h-h^\ast,u_n\rangle\right]\geq \inf_{u\in D}J_{h}(u)-\rho M.
	\end{equation}
	Let $\{v_n\}\subset \partial D$ be such that $J_{h}(v_n)\to\inf_{u\in \partial D}J_{h}(u)$ as $n\to\infty.$ Then
	$$\inf_{u\in \partial D}J_{h}(u)=\lim_{n\to\infty}\left[J_{h^\ast}(v_n)-\langle h-h^\ast,v_n\rangle\right]\geq \inf_{u\in \partial D}J_{h^\ast}(u)-\rho M.$$
	Combining this with \eqref{V.choose.rho} and \eqref{V.inf.D}, we obtain
	$$\inf_{u\in  D}J_{h}(u)<\inf_{u\in \partial D}J_{h}(u).$$
	Then, arguing as in the proof of Theorem~\ref{V.Theo.singular} we obtain the desired conclusions.
\end{proof}

\appendix
\section{Proof of Proposition~\ref{II.behavior.of.phi_1}} \label{AppendixA}

\begin{footnotesize}
\begin{proof}[Proof of Proposition~\ref{II.behavior.of.phi_1}]
	The radial symmetry of $\varphi_1$ follows from the radial symmetry of $K$ and the simplicity of $\lambda_1$ (see \cite[Proof of Theorem 1.1 (h)]{Chhetri-Drabek}). Moreover, $\varphi_1\in C^1(\overline{B_1^c})$, $\varphi_1>0$ in $B_1^c$ and $\varphi_1(|x|)\to 0$ as $|x|\to\infty$ due to \cite[Theorem 1.4]{Anoop.CV}. Thus, $\varphi_1$ satisfies
	\begin{equation*}\label{II.eq}
	\begin{cases}
	-\left(r^{N-1}|\varphi_1'|^{p-2}\varphi_1'\right)'=\lambda_1r^{N-1}K(r)\varphi_1^{p-1}\quad \text{in}\ (1,\infty),\\
	\varphi_1(1)=\varphi_1(\infty)=0.
	\end{cases}
	\end{equation*}
	Clearly, there is a unique $r_0\in (1,\infty)$ such that $\varphi'_1(r_0)=0$  and $\varphi_1'>0$ in $[1,r_0)$ and $\varphi_1'<0$ in $(r_0,\infty).$ So 
	we may define
	\begin{equation}\label{II.def.U}
	U(r):=r^{p-1}\left[\frac{-\varphi_1'(r)}{\varphi_1(r)}\right]^{p-1},\ \forall r\geq r_0.
	\end{equation}
	Hence $U(r_0)=0$ and $U(r)>0$ for all $r
	>r_0.$ Using the same argument as in \cite[Proof of Proposition 9.1]{Alziary} for all $r\geq r_0,$ we obtain
	\begin{equation}\label{II.est.U}
	U(r)\leq C_{N,p}:=\left(\frac{N-p}{p-1}\right)^{p-1}, 
	\end{equation}
	and
	\begin{equation}\label{II.eq.of.U}
	U'(r)=\frac{p-1}{r}U(r)\left(U(r)^{\frac{1}{p-1}}-\frac{N-p}{p-1}\right)+\lambda_1r^{p-1} K(r).	
	\end{equation}
	Now, we define
	\begin{equation}\label{II.appendix.form.a}
	a(r):=\frac{p-1}{r}\left(\frac{N-p}{p-1}-U(r)^{\frac{1}{p-1}}\right), \quad  \forall r\geq r_0,
	\end{equation}
	and for $r\geq t\geq r_0$, set
	$$A_t(r):=\int_{t}^{r}a(s)\diff s.$$
	Using \eqref{II.appendix.form.a} and \eqref{II.def.U}, we have 
	$$a(r)=\frac{N-p}{r}+(p-1)\frac{\varphi_1'(r)}{\varphi_1(r)}=(p-1)\frac{\diff}{\diff r}\log\left(r^{\frac{N-p}{p-1}}\varphi_1(r)\right), \quad \forall r\geq r_0,$$
	and hence
	\begin{equation}\label{II.form.A}
	A_t(r)=(p-1)\log\left(\frac{r^{\frac{N-p}{p-1}}\varphi_1(r)}{t^{\frac{N-p}{p-1}}\varphi_1(t)}\right),\quad \forall  t \in [r_0, r].
	\end{equation}
	Furthermore, we claim
	\begin{equation}\label{II.est.a}
	a(r)\geq 0,\ \forall r\geq r_0\ \ \text{and}\  A_{r_0}(\infty):=\int_{r_0}^\infty a(s)\diff s<\infty.
	\end{equation}
	Indeed, the positivity of $a(r)$ in $[r_0,\infty)$ is trivial by \eqref{II.est.U} and \eqref{II.appendix.form.a}. Hence, the function $r\mapsto A_{r_0}(r)$ is increasing in $[r_0,\infty).$ Suppose by contradiction $A_{r_0}(\infty)=\infty.$ Now, using 
	\eqref{II.eq.of.U}
and \eqref{II.appendix.form.a}, we have
	$$U'(r)+a(r)U(r)=\lambda_1r^{p-1}K(r)\quad \text{i.e.,}\quad \frac{\diff}{\diff r}\left(e^{\int_{t}^{r}a(\tau)\diff\tau}U(r)\right)=\lambda_1r^{p-1}K(r)e^{\int_{t}^{r}a(\tau)\diff\tau}.$$
	Hence, we obtain
	\begin{equation}\label{II.eq.of.U.2}
	U(r)-U(t)e^{-\int_{t}^{r}a(s)\diff s}=\lambda_1\int_{t}^{r}s^{p-1}K(s)e^{-\int_{s}^{r}a(\tau)\diff \tau}\diff s,\quad \forall t \in [r_0, r].
	\end{equation}
	Putting $t=r_0$ in (\ref{II.eq.of.U.2}) and recalling $U(r_0)=0$, we have 
	\begin{align}\label{II.eq.of.U.3}
	U(r)&=\lambda_1\int_{r_0}^{r}s^{p-1}K(s)e^{-\int_{s}^{r}a(\tau)\diff\tau}\diff s \notag \\&=\lambda_1\int_{r_0}^{\infty}s^{p-1}K(s)e^{-[A_{r_0}(r)-A_{r_0}(s)]}\chi_{(r_0,r)}(s)\diff s, \quad \forall r\geq r_0,
	\end{align}
	where $\chi_{(r_0,r)}$ is the characteristic function in $(r_0,r).$ Since $A_{r_0}(\infty)=\infty$ and the function $r\mapsto A_{r_0}(r)$ is increasing in $[r_0,\infty)$, we get
	$$s^{p-1}K(s)e^{-[A_{r_0}(r)-A_{r_0}(s)]}\chi_{(r_0,r)}(s)\to 0\ \text{for a.e.}\ s\in (r_0,\infty) \ \text{as}\ r\to\infty\ \text{and}$$
	$$\left|s^{p-1}K(s)e^{-[A_{r_0}(r)-A_{r_0}(s)]}\chi_{(r_0,r)}(s)\right|\leq s^{p-1}K(s)\in L^1(r_0,\infty),$$
	for a.e.\ $s\in (r_0,\infty)$ and for all $r\in(r_0,\infty).$ From this and  \eqref{II.eq.of.U.3}, we obtain 
	$\displaystyle\lim_{r\to\infty}U(r)=0$, via the Lebesgue dominated convergence theorem. Thus, for $\eta:=N-1-\delta\in (0,N-p)$ with $\delta$ taken from $\textrm{(H)}$ there is $r_\eta>r_0$ such that
	$$a(r)=\frac{p-1}{r}\left(\frac{N-p}{p-1}-U(r)^{\frac{1}{p-1}}\right)\geq \frac{N-p-\eta}{r},\ \forall r\geq r_\eta,$$
	and hence
	$$A_t(r)=\int_{t}^{r}a(s)\diff s\geq (N-p-\eta)\log(\frac{r}{t}),\ \forall r\geq t\geq r_\eta.$$
	Applying this estimate to \eqref{II.eq.of.U.2} with $t=r_\eta$, we obtain
	\begin{align}
	U(r)&=U(r_\eta)e^{-\int_{r_\eta}^{r}a(s)\diff s}+\lambda_1\int_{r_\eta}^{r}s^{p-1}K(s)e^{-\int_{s}^{r}a(\tau)\diff \tau}\diff s\notag\\
	&\leq U(r_\eta)(\frac{r}{r_\eta})^{-(N-p-\eta)}+\lambda_1\int_{r_\eta}^{r}s^{p-1}K(s)(\frac{r}{s})^{-(N-p-\eta)}\diff s\notag\\
	&= r^{p-1-\delta}\left[U(r_\eta){r_\eta}^{N-p-\eta}+\lambda_1\int_{r_\eta}^{r}s^{\delta}K(s)\diff s\right]\notag\\
	&\leq C_\eta r^{p-1-\delta},\ \forall r\geq r_\eta,      \label{U-estimate1}
	\end{align}
	where $C_\eta:=U(r_\eta){r_\eta}^{N-p-\eta}+\lambda_1\int_{r_\eta}^{\infty}s^{\delta}K(s)\diff s\in (0,\infty).$
	Combining \eqref{U-estimate1} with \eqref{II.def.U}, we obtain
	$$-\frac{\varphi_1'(r)}{\varphi_1(r)}\leq C_\eta^{\frac{1}{p-1}}r^{-\frac{\delta}{p-1}},\ \forall r\geq r_\eta,$$
	and hence
	$$-\log\left[\frac{\varphi_1(r)}{\varphi_1(r_\eta)}\right]\leq \frac{(p-1)C_\eta^{\frac{1}{p-1}}}{\delta-p+1}\left(r_\eta^{\frac{p-1-\delta}{p-1}}-r^{\frac{p-1-\delta}{p-1}}\right),\ \forall r\geq r_\eta.$$
	This contradicts to the fact that $\varphi_1(r)\to 0$ as $r\to\infty.$ So $A_{r_0}(\infty)<\infty$, and we have just proved \eqref{II.est.a}. Finally we show \eqref{II.behavior1} and \eqref{II.behavior2}. From \eqref{II.est.U}, we have
	$$U(\infty):=\lim_{r\to\infty}U(r)\leq C_{N,p}.$$
	If $U(\infty)< C_{N,p},$ then there exists $\gamma>0$ and $r_1>r_0$ such that
	$$a(r)=\frac{p-1}{r}\left(\frac{N-p}{p-1}-U(r)^{\frac{1}{p-1}}\right)\geq \frac{\gamma}{r},\ \forall r\geq r_1,$$
	a contradiction to $A_{r_0}(\infty)<\infty$. Thus, 
	\begin{equation}\label{II.lim.U}
	\lim_{r\to\infty}U(r)=C_{N,p}.
	\end{equation}
	It follows from \eqref{II.form.A} and \eqref{II.est.a} that
	$$\exists \lim_{r\to\infty}r^{\frac{N-p}{p-1}}\varphi_1(r)=:C\in (0,\infty),$$
	i.e., we get \eqref{II.behavior1}. Combining this with \eqref{II.lim.U} and \eqref{II.def.U} we get \eqref{II.behavior2} and the proof of Proposition~\ref{II.behavior.of.phi_1} is complete.
\end{proof}


\section{Proof of Lemma~\ref{II.le.compact.imb}}\label{AppendixB}
\begin{proof}[Proof of Lemma~\ref{II.le.compact.imb}]
	Let us first obtain the embeddings when we only assume that $p>2$ and $\textrm{(A)}$ hold. For each $u\in C_c^1(B_1^c)$, we have
	\begin{align*}
	\lambda_1\int_{B_1^c}K(x)\varphi_1^{p-2}u^2 \diff x&=\int_{B_1^c}(-\Delta_p\varphi_1)\varphi_1^{-1}u^2\diff x\\
	&=\int_{B_1^c}|\nabla\varphi_1|^{p-2}\nabla\varphi_1\cdot\nabla(\varphi_1^{-1}u^2)\diff x\\
	&=2\int_{B_1^c}|\nabla\varphi_1|^{p-2}\left(\nabla\varphi_1\cdot\nabla u\right) \varphi_1^{-1}u \diff x-\int_{B_1^c}|\nabla\varphi_1|^{p} \varphi_1^{-2} u^2 \diff x.
	\end{align*}
	Thus by using H\"older inequality, we get
	\begin{align}\label{II.est.norms}
	\lambda_1\int_{B_1^c}K(x)\varphi_1^{p-2}u^2\diff x&+\int_{B_1^c}|\nabla\varphi_1|^{p}\varphi_1^{-2} u^2\diff x\notag\\
	&\leq2\left(\int_{B_1^c}|\nabla\varphi_1|^{p-2}|\nabla u|^2\diff x\right)^{1/2}\left(\int_{B_1^c}|\nabla\varphi_1|^{p}\varphi_1^{-2}u^2 \diff x\right)^{1/2}.
	\end{align}
	From the density of $C_c^1(B_1^c)$ in $\mathcal{D}_{\varphi_1}$, we deduce from \eqref{II.est.norms} that for all $u\in \mathcal{D}_{\varphi_1}$, we have
	\begin{equation}\label{II.uphi}
	\int_{B_1^c}K(x)\varphi_1^{p-2}u^2\diff x\leq \frac{2}{\lambda_1}\|u\|_{\mathcal{D}_{\varphi_1}}\left(\int_{B_1^c}|\nabla \varphi_1|^{p}\varphi_1^{-2}u^2\diff x\right)^{1/2}, 
	\end{equation}
	and 
	\begin{equation}\label{II.uphi1}
	2\lambda_1\int_{B_1^c}K(x)\varphi_1^{p-2}u^2\diff x+\int_{B_1^c}|\nabla\varphi_1|^{p}\varphi_1^{-2}u^2\diff x\leq 4\int_{B_1^c}|\nabla\varphi_1|^{p-2}|\nabla u|^2\diff x. 
	\end{equation}
	Thus, we obtain $\mathcal{D}_{\varphi_1}\hookrightarrow \mathcal{H}_{\varphi_1}\ \ \text{and}\ \ \mathcal{D}_{\varphi_1}\hookrightarrow L^2(B_1^c;|\nabla \varphi_1|^p\varphi_1^{-2}).$
	
	Next, we show that the embedding $\mathcal{D}_{\varphi_1}\hookrightarrow\mathcal{H}_{\varphi_1}$ is compact if we assume in addition that $p<N$, $\textrm{(H)}$ and $\displaystyle\lim_{\rho\to\infty}\underset{r\geq\rho}{\operatorname{ess\ sup}}\ r^pK(r)=0$ hold. Let  $\psi_1:\ [0,\infty)\to [0,1]$ be any $C^1$ function such that $\psi_1(r)=1$ for $0\leq r\leq 1,\ \psi_1(r)=0$ for $2\leq r< \infty$ and $\psi_1'(r)\leq 0$ for $1\leq r\leq 2.$ For each $\rho>0$ we define
	$$\psi_\rho(x)=\psi_\rho(|x|):=\psi_1\left(|x|/\rho\right),\ \forall x\in \overline{B_1^c}.$$ Since $|\nabla\psi_\rho(x)|=\frac{1}{\rho}|\psi_1'(|x|/\rho)|,$ we deduce
	\begin{equation}\label{II.est.gradPsi_rho}
	|\nabla\psi_\rho(x)|\leq \frac{C_1}{|x|},\ \forall x\in B_1^c,
	\end{equation}
	where $C_1:=2\underset{1\leq r<\infty}{\sup}|\psi_1'(r)|.$
	Define $T_\rho(u):=\psi_\rho u$ for all $u\in\mathcal{D}_{\varphi_1}.$  We have
	\par \textbf{Claim 1.} $T_\rho: \mathcal{D}_{\varphi_1}\to \mathcal{D}_{\varphi_1}$ and there exist $C_2>0$ and $R_1>0$ such that for all $\rho\geq R_1,$
	\begin{equation}\label{II.appendix.uniform}
	\|T_\rho(u)\|_{\mathcal{D}_{\varphi_1}}\leq C_2\|u\|_{\mathcal{D}_{\varphi_1}},\quad \forall u\in\mathcal{D}_{\varphi_1}.
	\end{equation}
Indeed, using the Minkowski inequality we estimate
		\begin{align}\label{II.est.T_rho}
		\|T_\rho(u)\|_{\mathcal{D}_{\varphi_1}}&=\|\psi_\rho u\|_{\mathcal{D}_{\varphi_1}}=\left(\int_{B_1^c}|\nabla \varphi_1|^{p-2}|\nabla(\psi_\rho u)|^2\diff x\right)^{1/2}\notag\\
		&\leq\left(\int_{B_1^c}|\nabla \varphi_1|^{p-2}\psi_\rho^2|\nabla u|^2\diff x\right)^{1/2}+\left(\int_{B_1^c}|\nabla \varphi_1|^{p-2}|\nabla \psi_\rho|^2u^2\diff x\right)^{1/2}\notag\\
		&\leq\|u\|_{\mathcal{D}_{\varphi_1}}+\left(\int_{B_1^c}|\nabla \varphi_1|^{p-2}|\nabla \psi_\rho|^2u^2\diff x\right)^{1/2}.
		\end{align}
		Lemma~\ref{II.behavior.of.phi_1} yields
		\begin{equation}\label{II.est.phixphi'}
		\varphi_1^{-1}|\varphi_1'|\geq \frac{N-p}{2(p-1)r},\ \forall r\geq R_1,
		\end{equation}
		for $R_1>r_0$ sufficiently large. By this and \eqref{II.est.gradPsi_rho} we obtain
		\begin{equation}\label{II.est.1.phi}
		|\nabla\psi_\rho(x)|\leq M\varphi_1^{-1}(x)|\nabla\varphi_1(x)|,\ \forall |x|\geq R_1, 
		\end{equation}
		where $M:=\frac{2(p-1)C_1}{N-p}.$ Combining \eqref{II.uphi1}, \eqref{II.est.T_rho} and \eqref{II.est.1.phi}, we deduce
		\begin{align*}
		\|T_\rho(u)\|_{\mathcal{D}_{\varphi_1}}\leq\|u\|_{\mathcal{D}_{\varphi_1}}+M\left(\int_{B_1^c}|\nabla \varphi_1|^{p}\varphi_1^{-2}u^2\diff x\right)^{1/2}\leq (1+2M)\|u\|_{\mathcal{D}_{\varphi_1}},
		\end{align*}
		for all $\rho>R_1$ and for all $u\in\mathcal{D}_{\varphi_1}.$
		Thus, we obtain \eqref{II.appendix.uniform} with $C_2:= 1+2M$ and hence, Claim 1 is proved.
	\par Denoting by $J_{\varphi_1}$ the continuous embedding $\mathcal{D}_{\varphi_1}\hookrightarrow \mathcal{H}_{\varphi_1},$ we have
	\par \textbf{Claim 2.} $J_{\varphi_1}\circ T_\rho\to J_{\varphi_1}$ in the uniform operator topology as $\rho\to\infty,$ i.e.,
	$\|T_\rho(u)-u\|_{\mathcal{H}_{\varphi_1}}\to 0$  as $\rho\to\infty$ uniformly for $\|u\|_{\mathcal{D}_{\varphi_1}}\leq 1$.
	
	\noindent Indeed, 
	\eqref{II.est.phixphi'} yields
	$$\frac{K(r)\varphi_1^{p-2}(r)}{|\varphi_1'(r)|^p\varphi_1^{-2}(r)}\leq C_3r^pK(r),\ \forall r\geq R_1,$$
	where $C_3:=\left[\frac{2(p-1)}{N-p}\right]^p.$
	Combining this with \eqref{II.uphi1}, for all $u\in \mathcal{D}_{\varphi_1}$ and for all $\rho\geq R_1$, we get
	\begin{align*}
	4\|u\|_{\mathcal{D}_{\varphi_1}}^2\geq \int_{|x|\geq\rho}|\nabla\varphi_1|^p\varphi_1^{-2}u^2\diff x&\geq \int_{|x|\geq\rho}\frac{K(x)\varphi_1^{p-2}u^2}{C_3|x|^pK(x)}\diff x\\
	&\geq \frac{1}{C_3\underset{r\geq \rho}{\operatorname{ess\ sup}}\ r^pK(r)}\int_{|x|\geq\rho}K(x)\varphi_1^{p-2}u^2\diff x.
	\end{align*}
	Thus, Claim 2 is proved. 
	
 We know that the limit of a norm-convergent sequence of compact operators is also a compact operator. So by Claim 2, to show the compactness of the embedding $\mathcal{D}_{\varphi_1}\hookrightarrow\mathcal{H}_{\varphi_1},$ it suffices to show that $J_{\varphi_1}\circ T_\rho: \mathcal{D}_{\varphi_1}\to \mathcal{H}_{\varphi_1}$ is compact for $\rho>0$ sufficiently large. For $r>1,$ define
	$$\mathcal{D}_{\varphi_1}(A_1^r):=\{ u\in \mathcal{D}_{\varphi_1}: u=0\ \text{a.e. in}\ |x|\geq r\}.$$  
	Clearly, $\mathcal{D}_{\varphi_1}(A_1^r)$ is a closed linear subspace of $\mathcal{D}_{\varphi_1}.$ By Claim 1, \eqref{II.appendix.uniform} the mappings $T_\rho: \mathcal{D}_{\varphi_1}\to \mathcal{D}_{\varphi_1}(A_1^{2\rho})\subset \mathcal{D}_{\varphi_1}$ are uniformly bounded for all $\rho\geq R_1.$ To show that $J_{\varphi_1}\circ T_\rho: \mathcal{D}_{\varphi_1}\to \mathcal{H}_{\varphi_1}$ is compact, it suffices to show that  $\mathcal{D}_{\varphi_1}(A_1^{2\rho})\hookrightarrow \mathcal{H}_{\varphi_1}$ is compact. Before doing this, we obtain the following estimate:
	\begin{equation}\label{II.uphi2}
	\int_{A_R^{R'}}|\nabla \varphi_1|^p\varphi_1^{-2}u^2\diff x\leq 9\log\left(\frac{\varphi^2_1(r_0)}{\varphi_1(R)\varphi_1(R')}\right)\|u\|^2_{\mathcal{D}_{\varphi_1}},\quad \forall 1<R<r_0<R', \forall u\in\mathcal{D}_{\varphi_1}.
	\end{equation}	
	Clearly, the estimate \eqref{II.uphi2} is immediately obtained if we can prove that for all $1<R<r_0<R'$ we have
	\begin{equation}\label{II.log1}
	\int_{A_{r_0}^{R'}}|\nabla \varphi_1|^p\varphi_1^{-2}u^2\diff x\leq 9\log\left(\frac{\varphi_1(r_0)}{\varphi_1(R')}\right)\|u\|^2_{\mathcal{D}_{\varphi_1}},
	\end{equation}
	and 
	\begin{equation}\label{II.log1'}
	\int_{A_{R}^{r_0}}|\nabla \varphi_1|^p\varphi_1^{-2}u^2\diff x\leq 9\log\left(\frac{\varphi_1(r_0)}{\varphi_1(R)}\right)\|u\|^2_{\mathcal{D}_{\varphi_1}}.
	\end{equation}
	To obtain \eqref{II.log1}, we proceed as in \cite[Proof of (4.16)]{Alziary}. Fix any $x'\in \mathbb{R}^N$ with $|x'|=1,$ and take $x=rx',\ r_0\leq r\leq R'.$ We have
	\begin{align*}
	&r^{N-1}|\varphi_1'(r)|^{p-1}\varphi^{-1}_1(r)u^2(rx')=-r^{N-1}\left(|\varphi_1'(r)|^{p-2}\varphi_1'(r)\right)\varphi_1^{-1}(r)u^2(rx')\\
	&=-\int_{r_0}^{r}\frac{\partial}{\partial s}\left[s^{N-1}\left(|\varphi_1'(s)|^{p-2}\varphi_1'(s)\right)\varphi_1^{-1}(s)u^2(sx')\right]\diff s\\
	&=\lambda_1\int_{r_0}^{r}s^{N-1}K(s)\varphi^{p-2}_1(s)u^2(sx')\diff s+\int_{r_0}^{r}s^{N-1}|\varphi_1'(s)|^{p}\varphi_1^{-2}(s)u^2(sx')\diff s\\
	&\hspace*{3cm} +2\int_{r_0}^{r}s^{N-1}|\varphi_1'(s)|^{p-1}\varphi_1^{-1}(s)u(sx')\frac{\partial u}{\partial s}(sx')\diff s.
	\end{align*}
Using the Cauchy-Schwarz inequality and then using the Cauchy inequality for the last integral we get from the preceding equality that
	\begin{align*}
	r^{N-1}|\varphi_1'(r)|^{p-1}&\varphi_1^{-1}(r)u^2(rx')\leq\lambda_1\int_{r_0}^{r}K(s)s^{N-1}\varphi_1^{p-2}(s)u^2(sx')\diff s\\
	&+2\int_{r_0}^{r}s^{N-1}|\varphi_1'(s)|^{p}\varphi_1^{-2}(s)u^2(sx')\diff s+\int_{r_0}^{r}s^{N-1}|\varphi_1'(s)|^{p-2}\left(\frac{\partial u}{\partial s}(sx')\right)^2\diff s.
	\end{align*}
	By integrating with respect to $x'$ over the unit sphere $S_1=\partial B_1\subset \mathbb{R}^N$ endowed with the surface measure $\diff\sigma$ and then changing variable $y=sx'$ we obtain from the last inequality and \eqref{II.uphi1} that
	\begin{align*}
	r^{N-1}|\varphi_1'(r)|^{p-1}\varphi^{-1}_1(r)&\int_{S_1}u^2(rx')\diff\sigma(x')\leq\lambda_1\int_{A_{r_0}^r}K(y)\varphi_1^{p-2}(y)u^2(y)\diff y\\
	&+2\int_{A_{r_0}^r}|\nabla\varphi_1(y)|^{p}\varphi_1^{-2}(y)u^2(y)\diff y+\int_{A_{r_0}^r}|\nabla\varphi_1|^{p-2}\left(\frac{\partial u}{\partial s}\right)^2\diff y\leq 9\|u\|_{\mathcal{D}_{\varphi_1}}.
	\end{align*}
	This yields
	$$\int_{r_0}^{R'}r^{N-1}|\varphi_1'(r)|^{p}\varphi^{-2}_1(r)\int_{S_1}u^2(rx')\diff\sigma(x')\diff r\leq -9\|u\|_{\mathcal{D}_{\varphi_1}}\int_{r_0}^{R'}\frac{\varphi_1'(s)}{\varphi_1(s)}\diff s,$$
	and hence \eqref{II.log1} follows.

 Let us prove \eqref{II.log1'}. For $R\leq r\leq r_0,$ we have
	\begin{align*}
	&r^{N-1}|\varphi_1'(r)|^{p-1}\varphi^{-1}_1(r)u^2(rx')=r^{N-1}\left(|\varphi_1'(r)|^{p-2}\varphi_1'(r)\right)\varphi_1^{-1}(r)u^2(rx')\\
	&=-\int_{r}^{r_0}\frac{\partial}{\partial s}\left[s^{N-1}\left(|\varphi_1'(s)|^{p-2}\varphi_1'(s)\right)\varphi_1^{-1}(s)u^2(sx')\right]\diff s\\
	&=\lambda_1\int_{r}^{r_0}s^{N-1}K(s)\varphi^{p-2}_1(s)u^2(sx')\diff s+\int_{r}^{r_0}s^{N-1}|\varphi_1'(s)|^{p}\varphi_1^{-2}(s)u^2(sx')\diff s\\
	&\hspace*{3cm} -2\int_{r}^{r_0}s^{N-1}|\varphi_1'(s)|^{p-1}\varphi_1^{-1}(s)u(sx')\frac{\partial u}{\partial s}(sx')\diff s.
	\end{align*}
Arguing as above, we deduce \eqref{II.log1'}.

 Now we prove that $\mathcal{D}_{\varphi_1}(A_1^{2\rho})\hookrightarrow\hookrightarrow\mathcal{H}_{\varphi_1}.$ Indeed, let $u_n\rightharpoonup 0$ in $\mathcal{D}_{\varphi_1}(A_1^{2\rho})$ as $n\to\infty$. Then $\{u_n\}$ is bounded in $\mathcal{D}_{\varphi_1}(A^{2\rho}_1)$. Without loss of generality we assume that $\|u_n\|_{\mathcal{D}_{\varphi_1}} \leq 1$ for all $n\in \mathbb{N}$.
 Next we show that $u_n\to 0$ strongly in $\mathcal{H}_{\varphi_1}$ as $n\to\infty$. Let $\epsilon>0,$ and $1<R<r_0<R'<2\rho$ be such that
	\begin{equation}\label{II.log2}
	9\log\left(\frac{\varphi^2_1(r_0)}{\varphi_1(R)\varphi_1(R')}\right)\leq\left(\frac{\lambda_1\epsilon}{8}\right)^2.
	\end{equation} 
	Let $\delta>0$ be such that $R+\delta<r_0<R'-\delta.$ We have
	\begin{align}\label{II.appendix.int1}
	&\int_{B_1^c}K(x)\varphi_1^{p-2}u_n^2\diff x=\int_{A_1^{R+\delta}}K(x)\varphi_1^{p-2}u_n^2\diff x+\int_{A_{R+\delta}^{R'-\delta}}K(x)\varphi_1^{p-2}u_n^2\diff x\notag\\
	&\hspace*{6cm}+\int_{A_{R'-\delta}^{2\rho}}K(x)\varphi_1^{p-2}u_n^2\diff x\notag\\
	&\leq \|\varphi_1\|_{L^\infty(B_1^c)}^{p-2}\|K\|_{L^\infty(B_1^c)}\left(\int_{A_1^{R+\delta}}u_n^2\diff x+\int_{A_{R'-\delta}^{2\rho}}u_n^2\diff x\right)+\int_{A_{R+\delta}^{R'-\delta}}K(x)\varphi_1^{p-2}u_n^2\diff x.
	\end{align}
	Let $\Omega=A_1^{R+\delta}$ or 	$\Omega=A_{R'-\delta}^{2\rho}$. Since $\mathcal{D}_{\varphi_1}\hookrightarrow L^2(B_1^c;|\nabla\varphi_1|^p\varphi_1^{-2})$ and $\underset{\Omega}{\inf}|\nabla \varphi_1|>0$, we have 
	$$\int_{\Omega}|\nabla u_n|^2\diff x\leq \frac{1}{(\underset{\Omega}{\inf}|\nabla \varphi_1|)^{p-2}}\int_{\Omega}|\nabla\varphi_1|^{p-2}|\nabla u_n|^2\diff x\leq\frac{1}{(\underset{\Omega}{\inf}|\nabla \varphi_1|)^{p-2}}\|u_n\|^2_ {\mathcal{D}_{\varphi_1}},$$
	$$\int_{\Omega}u_n^2\diff x\leq C_1(\Omega)\int_{B_1^c}|\nabla\varphi_1|^p\varphi_1^{-2}u_n^2\diff x\leq C_2(\Omega)\|u_n\|^2_{\mathcal{D}_{\varphi_1}},$$
	where $C_1(\Omega)$ and $C_2(\Omega)$ are positive constants independent of $n$. Hence, $\{u_n\}$ is bounded in $W^{1,2}(\Omega)$ and thus up to a subsequence, $u_n\rightharpoonup 0$ in $W^{1,2}(\Omega)$ as $n\to\infty$. So we get $u_n\to 0$ in $L^2(\Omega)$ as $n\to\infty$ and therefore, there exists $n_1\in\mathbb{N}$ such that
	\begin{equation}\label{II.appendix.e1}
	\|\varphi_1\|_{L^\infty(B_1^c)}^{p-2}\|K\|_{L^\infty(B_1^c)}\left(\int_{A_1^{R+\delta}}u_n^2\diff x+\int_{A_{R'-\delta}^{2\rho}}u_n^2\diff x\right)\leq \frac{\epsilon}{2},\quad \forall n\geq n_1.
	\end{equation}
 Let $\phi\in C_c^\infty(B_1^c)$ satisfy $0\leq \phi\leq 1,$ $\phi\equiv 1$ in $A_{R+\delta}^{R'-\delta}$ and $\phi\equiv 0$ in $B_1^c\setminus A_R^{R'}.$ Then, $\phi u_n\in \mathcal{D}_{\varphi_1}$ for all $n$ in view of the embedding $\mathcal{D}_{\varphi_1}\hookrightarrow L^2(B_1^c;|\nabla\varphi_1|^p\varphi_1^{-2}).$ 
 So we apply \eqref{II.uphi} for $u=\phi u_n$ and use the Minkowski inequality, to get
	\begin{align}\label{II.appendix.int2}
	\int_{A_{R+\delta}^{R'-\delta}}K(x)\varphi_1^{p-2}u_n^2\diff x &\leq \int_{B_1^c}K(x)\varphi_1^{p-2}(\phi u_n)^2\diff x\notag\\ &\leq \frac{2}{\lambda_1}\|\phi u_n\|_{\mathcal{D}_{\varphi_1}}\left(\int_{B_1^c}|\nabla\varphi_1|^p\varphi_1^{-2}(\phi u_n)^2\diff x\right)^{\frac{1}{2}}\notag\\
	&\leq\frac{2}{\lambda_1}\left(\int_{B_1^c}|\nabla\varphi_1|^{p-2}|\phi\nabla u_n+u_n\nabla\phi|^2\diff x\right)^{\frac{1}{2}}\left(\int_{A_R^{R'}}|\nabla\varphi_1|^{p}\varphi_1^{-2}u_n^2\diff x\right)^\frac{1}{2}\notag\\
	&\leq\frac{2}{\lambda_1}\left[\left(\int_{B_1^c}|\nabla\varphi_1|^{p-2}|\nabla u_n|^2\diff x\right)^{\frac{1}{2}}+\left(\int_{B_1^c}|\nabla\phi|^2|\nabla\varphi_1|^{p-2}u_n^2\diff x\right)^{\frac{1}{2}}\right] \notag \\  &\qquad \qquad \times \left(\int_{A_R^{R'}}|\nabla\varphi_1|^{p}\varphi_1^{-2}u_n^2\diff x\right)^\frac{1}{2}\notag\\
	&\leq\frac{2}{\lambda_1}\left[1+\|\nabla\phi\|_{L^\infty(B_1^c)}\|\nabla\varphi_1\|_{L^\infty(B_1^c)}^{p-2}\left(\int_{A_1^{R+\delta}}u_n^2\diff x+\int_{A_{R'-\delta}^{2\rho}}u_n^2\diff x\right)^{\frac{1}{2}}\right] \notag\\  &\qquad \qquad \times \left(\int_{A_R^{R'}}|\nabla\varphi_1|^{p}\varphi_1^{-2}u_n^2\diff x\right)^\frac{1}{2}.
	\end{align}
	Applying \eqref{II.uphi2} for $u=u_n$ and invoking \eqref{II.log2}, we obtain
	\begin{equation}\label{II.appendix.e2}
	\left(\int_{A_R^{R'}}|\nabla\varphi_1|^{p}\varphi_1^{-2}u_n^2\diff x\right)^\frac{1}{2}\leq \frac{\lambda_1\epsilon}{8}.
	\end{equation}
	Since $\int_{A_1^{R+\delta}}u_n^2\diff x+\int_{A_{R'-\delta}^{2\rho}}u_n^2\diff x\to 0$ as $n\to\infty$, we can find $n_2\geq n_1$ such that
	$$1+\|\nabla\phi\|_{L^\infty(B_1^c)}\|\nabla\varphi_1\|_{L^\infty(B_1^c)}^{p-2}\left(\int_{A_1^{R+\delta}}u_n^2\diff x+\int_{A_{R'-\delta}^{2\rho}}u_n^2\diff x\right)^{\frac{1}{2}}<2,\quad \forall n\geq n_2.$$
	Combining this with \eqref{II.appendix.int2} and \eqref{II.appendix.e2} we deduce that
	$$\int_{A_{R+\delta}^{R'-\delta}}K(x)\varphi_1^{p-2}u_n^2\diff x\leq\frac{\epsilon}{2},\quad \forall n\geq n_2.$$
	By this, \eqref{II.appendix.int1}, and \eqref{II.appendix.e1} we obtain
	$$\int_{B_1^c}K(x)\varphi_1^{p-2}u_n^2\diff x<\epsilon,\quad \forall n\geq n_2.$$
	Since $\epsilon >0$ was arbitrary, it follows that $u_n \to 0$ in $\mathcal{H}_{\varphi_1}$. Thus, the embedding $\mathcal{D}_{\varphi_1}(A_1^{2\rho})\hookrightarrow\mathcal{H}_{\varphi_1}$ is compact and so is the embedding $\mathcal{D}_{\varphi_1} \hookrightarrow \mathcal{H}_{\varphi_1}$. The proof  of Lemma~\ref{II.le.compact.imb} is complete.
	\end{proof}


\section{Proof of Lemma~\ref{II.le.weak.derivative}}\label{Appendix.WeakDerivative}
\begin{proof}[Proof of Lemma~\ref{II.le.weak.derivative}]
	 Let $u\in \mathcal{D}_{\varphi_1}$ be arbitrary. By the definition of $\mathcal{D}_{\varphi_1}$, $u$ is represented by a sequence $\{u_n\}\subset X$ that is a Cauchy sequence with respect to the norm $\|\cdot\|_{\mathcal{D}_{\varphi_1}}$ i.e.,
	$$\int_{B_1^c}|\nabla \varphi_1|^{p-2}|\nabla u_n-\nabla u_m|^2\diff x\to 0\quad\text{as}\quad m,n\to\infty.$$
	By this and the embedding $\mathcal{D}_{\varphi_1}\hookrightarrow \mathcal{H}_{\varphi_1}$ we deduce that there is a measurable vector-valued function $v=(v_1,\cdots,v_N)$ such that $u_n\to u,$ $\nabla u_n\to v$ a.e. in $B_1^c,$ and
	\begin{equation}\label{II.A.WeakDerivative.convergence}
	\int_{B_1^c}|\nabla \varphi_1|^{p-2}|\nabla u_n-v|^2\diff x\to 0,\quad \int_{B_1^c}K\varphi_1^{p-2}|u_n-u|^2\diff x\to 0\quad \text{as}\quad n\to \infty. 
	\end{equation}
	As shown in the proof of Proposition~\ref{II.behavior.of.phi_1} we have $\varphi_1(x)=\varphi_1(|x|)$, $\varphi_1\in C^1(1,\infty),$  $\varphi_1>0$ in $(1,\infty)$ and there is a unique $r_0\in (1,\infty)$ such that $\varphi'_1(r_0)=0.$ Since $$-\left(r^{N-1}|\varphi_1'|^{p-2}\varphi_1'\right)'=\lambda_1r^{N-1}K(r)\varphi_1^{p-1}\quad \text{in}\ (1,\infty),$$
	we deduce that for each $n\geq 3$, there exists $C_n>0$ such that
	\begin{equation}\label{II.A.WeakDerivative1}
	|\varphi_1'(r)|^{2-p}\leq C_n\left|\int_{r_0}^{r}K(s)\diff s\right|^{\frac{2-p}{p-1}},\quad \forall r\in [1+\frac{1}{n},n].
	\end{equation}
	
	Let $\phi \in C_c^\infty(B_1^c)$ and let $n_0\in\mathbb{N}$ such that $\operatorname{supp}(\phi)\subset A_{1+1/n_0}^{n_0}.$ For each $i\in\{1,\cdots,N\}$ and for any $n\in\mathbb{N}$ we have
	\begin{equation}\label{II.A.weakderivative.form}
	\int_{B_1^c}\frac{\partial u_n}{\partial x_i}\phi\diff x=-\int_{B_1^c}u_n\frac{\partial \phi}{\partial x_i}\diff x.   
	\end{equation}
	Invoking \eqref{II.A.WeakDerivative1}, we estimate
	\begin{align*}
	&\left|\int_{B_1^c}\frac{\partial u_n}{\partial x_i}\phi\diff x-\int_{B_1^c}v_i\phi\diff x\right|\\
	&\leq \int_{B_1^c}|\phi||\nabla\varphi_1|^{\frac{2-p}{2}}|\nabla\varphi_1|^{\frac{p-2}{2}}\left|\frac{\partial u_n}{\partial x_i}-v_i\right|\diff x	\\
	&\leq \left(\int_{B_1^c}|\phi|^2|\nabla\varphi_1|^{2-p}\diff x \right)^{\frac{1}{2}}\left(\int_{B_1^c}|\nabla\varphi_1|^{p-2}\left|\frac{\partial u_n}{\partial x_i}-v_i\right|^2\diff x\right)^{\frac{1}{2}}\\
	&
	\leq C(n_0,\|\phi\|_{L^\infty(B_1^c)})\left(\int_{1+\frac{1}{n_0}}^{n_0}\left|\int_{r_0}^{r}K(s)\diff s\right|^{\frac{2-p}{p-1}}\diff r\right)^{\frac{1}{2}}\left(\int_{B_1^c}|\nabla\varphi_1|^{p-2}\left|\frac{\partial u_n}{\partial x_i}-v_i\right|^2\diff x\right)^{\frac{1}{2}}.
	\end{align*}
	Letting $n\to\infty$ in the preceding estimate, using $\textrm{(W)}$ and \eqref{II.A.WeakDerivative.convergence}, we obtain
	\begin{equation}\label{II.A.WeakDerivative.lim1}
	\lim_{n\to\infty}\int_{B_1^c}\frac{\partial u_n}{\partial x_i}\phi\diff x=\int_{B_1^c}v_i\phi\diff x.
	\end{equation}
	On the other hand, we have
	\begin{align*}
	\left|\int_{B_1^c}u_n\frac{\partial \phi}{\partial x_i}\diff x-\int_{B_1^c}u\frac{\partial \phi}{\partial x_i}\diff x\right|&\leq \int_{B_1^c}K^{-\frac{1}{2}}\varphi_1^{\frac{2-p}{2}}\left|\frac{\partial \phi}{\partial x_i}\right|K^{\frac{1}{2}}\varphi_1^{\frac{p-2}{2}}|u_n-u|\diff x	\\
	&\leq \|\nabla\phi\|_{L^\infty(B_1^c)}\left(\int_{1+\frac{1}{n_0}}^{n_0}K^{-1}\diff x\right)^{\frac{1}{2}}\int_{B_1^c}K\varphi_1^{p-2}|u_n-u|^2\diff x.	
	\end{align*}	
	Letting $n\to\infty$ in the preceding estimate, using $\textrm{(W)}$ and \eqref{II.A.WeakDerivative.convergence} again, we obtain
	\begin{equation*}
	\lim_{n\to\infty}\int_{B_1^c}u_n\frac{\partial \phi}{\partial x_i}\diff x=\int_{B_1^c}u\frac{\partial \phi}{\partial x_i}\diff x.
	\end{equation*}	
	Combining this with \eqref{II.A.WeakDerivative.lim1} and \eqref{II.A.weakderivative.form} we get
	\begin{equation*}
	\int_{B_1^c}v_i\phi\diff x=-\int_{B_1^c}u\frac{\partial \phi}{\partial x_i}\diff x,\quad \forall\phi\in C_1^\infty(B_1^c).
	\end{equation*}	
	Thus, $v=\nabla u$ and $u\in W^1(B_1^c),$ and the proof is completed.
\end{proof}

\section{Proof of Proposition~\ref{II.prop.simplicity}}\label{AppendixC}
\begin{proof}[Proof of Proposition~\ref{II.prop.simplicity}]
	Let $u\in\mathcal{D}_{\varphi_1}\setminus\{0\}$ be such that $\mathcal{Q}_0(u,u)=0.$ 
	Then, $u$ is a minimizer for $\lambda_1$ in \eqref{2nd.form.of.lamda1}. If $u$ changes sign in $B_1^c$ then $u^+\not\equiv 0,\ u^-\not\equiv 0,$ and
	\begin{align*}
	\lambda_1&=\frac{\int_{B_1^c}K(x)\varphi_1^{p-2}(u^+)^2\diff x}{\int_{B_1^c}K(x)\varphi_1^{p-2}u^2\diff x}\frac{\int_{B_1^c}\langle\mathbb{A}(\nabla\varphi_1)\nabla u^+,\nabla u^+\rangle_{\mathbb{R}^N}\diff x}{\int_{B_1^c}(p-1)K(x)\varphi_1^{p-2}(u^+)^2\diff x}\\
	&\qquad+\frac{\int_{B_1^c}K(x)\varphi_1^{p-2}(u^-)^2\diff x}{\int_{B_1^c}K(x)\varphi_1^{p-2}u^2\diff x}\frac{\int_{B_1^c}\langle\mathbb{A}(\nabla\varphi_1)\nabla u^-,\nabla u^-\rangle_{\mathbb{R}^N}\diff x}{\int_{B_1^c}(p-1)K(x)\varphi_1^{p-2}(u^-)^2\diff x}\\
	&\geq\left(\frac{\int_{B_1^c}K(x)\varphi_1^{p-2}(u^+)^2\diff x}{\int_{B_1^c}K(x)\varphi_1^{p-2}u^2\diff x}+\frac{\int_{B_1^c}K(x)\varphi_1^{p-2}(u^-)^2\diff x}{\int_{B_1^c}K(x)\varphi_1^{p-2}u^2\diff x}\right)\lambda_1=\lambda_1.
	\end{align*}
Thus, $u^+,u^-$ are 
minimizers for $\lambda_1$ in \eqref{2nd.form.of.lamda1}. 
Note that if $\mathcal{Q}_0(u,u)=0$ then $\mathcal{Q}_0(u,v)=0$ for all $v\in\mathcal{D}_{\varphi_1}.$ 
For each $t\in\mathbb{R},$ set $v_t:=u-t\varphi_1.$ If $v_t\ne 0,$ then since 
	$$\mathcal{Q}_0(\pm v_t,\pm v_t)=\mathcal{Q}_0(u,u)- 2\mathcal{Q}_0(u,\varphi_1)t+\mathcal{Q}_0(\varphi_1,\varphi_1)t^2=0,$$
	we have $\pm v_t$ are also minimizers for $\lambda_1$ in \eqref{2nd.form.of.lamda1}. If $v_t$ changes sign then $v_t^{\pm}$ are 
	 minimizers for $\lambda_1$ in \eqref{2nd.form.of.lamda1}. Then
	\begin{equation}\label{II.A.positiveness}
	-\nabla\cdot(\mathbb{A}(\nabla\varphi_1)\nabla v_t^{\pm})=\lambda_1(p-1)K\varphi_1^{p-2}v_t^{\pm}\geq 0 \ \ \text{in}\ B_1^c.
	\end{equation}
	So, in any case, $v_t^{\pm}$ are nonnegative solutions of \eqref{II.A.positiveness}. We know, $\varphi_1\in C^1(B_1^c),$ $\varphi_1(x)=\varphi_1(|x|)$ and $\varphi_1'(r)=0$ has a unique solution $r_0\in(1,\infty).$ Morever, it is easy to see that $\varphi_1\in W_{loc}^{2,\infty}(B_1^c\setminus S_{r_0}).$ Let $n_0\in \mathbb{N}$ be such that $1+\frac{1}{n_0}<r_0<n_0-\frac{1}{n_0}$ and set $\Omega_n:=A_1^{r_0-\frac{1}{n}}$ or $A_{r_0+\frac{1}{n}}^n$ $(n=n_0,n_0+1,\cdots)$. For each $n\geq n_0$ we have
	$$\mathbb{A}(\nabla\varphi_1)\xi\cdot\xi\geq |\nabla\varphi_1|^{p-2}|\xi|^2\geq \left(\underset{\Omega_{n}}{\inf}|\nabla \varphi_1|\right)^{p-2}|\xi|^2,\quad \forall \xi\in\mathbb{R}^N.$$
Thus, we can apply \cite[Theorem 8.22]{Gilbarg} and then \cite[Theorem 2 and its Remark]{DiBenedetto} to get $v_t^\pm\in C^{1,\beta_{n}}(\overline{\Omega}_{n})$, for some $\beta_{n}\in (0,1).$ Since $B_1^c\setminus S_{r_0}=\bigcup_{n=n_0}^\infty\left( A_1^{r_0-1/n}\cup A_{r_0+1/n}^n\right),$ we obtain that $v_t^{\pm}\in C^1(B_1^c\setminus S_{r_0}).$ By the strong maximum principle due to V\'{a}zquez \cite[Theorem 4, p.199]{Vazquez}, we have $v_t^{\pm}\equiv 0$ in $\Omega_n$ or else $v_t^{\pm}>0$ in $\Omega_n$ for each $n\geq n_0.$ Hence, we have $v_t^{\pm}\equiv 0$ in $A_1^{r_0}$ (resp. $B_{r_0}^c$) or else $v_t^{\pm}>0$ in $A_1^{r_0}$ (resp. $B_{r_0}^c$). 

Next, set $t_1=\frac{u(x_1)}{\varphi_1(x_1)}$ and $t_2=\frac{u(x_2)}{\varphi_1(x_2)}$ for some $x_1\in A_1^{r_0}$ and $x_2\in B_{r_0}^c$ then $v_{t_1}(x_1)=v_{t_2}(x_2)=0.$ We consider the following cases.
	\begin{itemize}
		\item[(i)] If $v_{t_1}^+\not \equiv 0$ and $v_{t_1}^-\not \equiv 0$ then $v_{t_1}^+\not \equiv 0$ and $v_{t_1}^-\not \equiv 0$ in $B_{r_0}^c$, since $v_{t_1}^+=v_{t_1}^-\equiv 0$ in $A_1^{r_0}$, and that is a contradiction. 
		\item[(ii)] If $v_{t_1}^+\not \equiv 0$ and $v_{t_1}^- \equiv 0$ then $v_{t_1}=v_{t_1}^+\geq 0$ satisfies
		$$-\nabla\cdot(\mathbb{A}(\nabla\varphi_1)\nabla v_{t_1})=\lambda_1(p-1)K\varphi_1^{p-2}v_{t_1}\geq 0 \ \ \text{in}\ B_1^c.$$
		Thus, $v_{t_1}\equiv 0$ in $A_1^{r_0}$ and $v_{t_1}> 0$ in $B_{r_0}^c.$ This yields $t_2=\frac{v_{t_1}(x_2)}{\varphi_1(x_2)}+t_1>t_1$ and hence, $v_{t_2}=(t_1-t_2)\varphi_1<0$ in $A_1^{r_0}.$ Because of this, $v_{t_2}^+\equiv 0$ in $A_1^{r_0}$ and therefore, $v_{t_2}^+\equiv 0$ in $B_1^c$ due to $v_{t_2}^+\equiv 0$ in $B_{r_0}^c.$ So $-v_{t_2}=v_{t_2}^-\geq 0$ satisfies
		$$-\nabla\cdot(\mathbb{A}(\nabla\varphi_1)\nabla (-v_{t_2}))=\lambda_1(p-1)K\varphi_1^{p-2}(-v_{t_2})\geq 0 \ \ \text{in}\ B_1^c.$$
		Thus, $-v_{t_2}\equiv 0$ in $B_{r_0}^c$ and hence, 
		\begin{equation}\label{II.form_u}
		u=\begin{cases}
		t_1\varphi_1\ \ \text{in}\ \ A_1^{r_0},\\ 
		t_2\varphi_1\ \ \text{in}\ \ B_{r_0}^c.
		\end{cases}
		\end{equation}
		Since $u\in\mathcal{D}_{\varphi_1},$ $u$ has a weak derivative on $A_1^{2r_0}$ in view of Lemma~\ref{II.le.weak.derivative} and by \eqref{II.form_u}, $u\in W^{1,\infty}(A_1^{2r_0}).$ A classical result (see e.g., \cite[Theorem 5 of Subsection 4.2.3]{Evans}) therefore let us know that $u$ is Lipschitz continuous in $A_{\frac{r_0+1}{2}}^{r_0+1}$ and this is impossible due to the form \eqref{II.form_u} of $u.$ 
		\item[(iii)] If $v_{t_1}^+\equiv 0$ and $v_{t_1}^- \not \equiv 0$ then as in the case (ii) we obtain $v_{t_1}\equiv 0$ in $A_1^{r_0}$ and $v_{t_2}\equiv 0$ in $B_{r_0}^c$ and this leads to a contradiction.
	\end{itemize}
	So we must have $v_{t_1}^+=v_{t_1}^-\equiv 0$ in $B_1^c,$ i.e., $u=t_1\varphi_1$ in $B_1^c$ and the proof is complete.
	\end{proof}
\section{Proof of Lemma~\ref{III.le.phi1.in.closure.of.Y}}\label{AppendixD}
\begin{proof}[Proof of Lemma~\ref{III.le.phi1.in.closure.of.Y}]
Let $\epsilon>0$ be arbitrary. We will show that there exists $w_\epsilon\in Y$ such that 
\begin{equation}\label{III.approximation}
\|w_\epsilon-\varphi_1\|<\epsilon.
\end{equation}
	We proceed in three steps. For the sake of brevity we will only give the details of Step 1, and we omit the details of Steps 2 and 3, since they essentially follow the same path of Step 1. Note that, as shown in the proof of Proposition~\ref{II.behavior.of.phi_1}, we have  $\varphi_1\in C^1[1,\infty),$ $\varphi_1(1)=0,$ $\varphi_1(r)>0$ for all $r>1$, $\varphi'_1>0$ and $\varphi'_1<0$ in $[1,r_0)$ and $(r_0,\infty),$ respectively for some $r_0\in (1,\infty).$

	\textit{Step 1: cut-off by zero near infinity.} By Proposition~\ref{II.behavior.of.phi_1}, there exist $C_1>0$ and $r_1>r_0$ such that
	\begin{equation}\label{III.est.phi1}
	\varphi_1(r)\leq\frac{C_1}{r^{\frac{N-p}{p-1}}},\ |\varphi'_1(r)|\leq \frac{C_1}{r^{\frac{N-1}{p-1}}},\ \left|\frac{\varphi'_1(r)}{\varphi_1(r)}\right|\leq \frac{C_1}{r},\ \forall r>r_1.
	\end{equation}
	For each $n\in\mathbb{N}\cap (r_1,\infty),$ let $\phi_{1,n}\in C^\infty(\mathbb{R})$ be such that
	\begin{equation*}\label{III.form.phi4,n}
	\phi_{1,n}(r)=\begin{cases}
	\varphi_1(n),\quad r\in (-\infty,n],\\
	0,\quad r\in [2n,\infty),\\
	\end{cases}
	\end{equation*}
	and for all $r\in [n,2n],$
	\begin{equation}\label{III.est.phi4,n}
	|\phi_{1,n}(r)|\leq \varphi_1(n)\leq\frac{C_1}{n^{\frac{N-p}{p-1}}},\ |\phi_{1,n}'(r)|\leq \gamma_0\frac{\varphi_1(n)}{n}\leq \gamma_0\frac{C_1}{n^{\frac{N-1}{p-1}}},
	\end{equation}
	where $\gamma_0$ is a positive constant independent of $n$. We then find $\alpha_{1,n}\in C^\infty(\mathbb{R})$ such that $\psi_{1,n}:=\phi_{1,n}\alpha_{1,n}$ satisfies
	\begin{equation}\label{III.cond.psi1,n}
	\begin{cases}
	\psi_{1,n}(n)=\varphi_1(n),\psi_{1,n}'(n)=\varphi_1'(n),\\
	\psi_{1,n}(2n)=\psi_{1,n}'(2n)=0.
	\end{cases}
	\end{equation}
	This is equivalent to
	\begin{equation} 
	\alpha_{1,n}(n)=1,\ \alpha'_{1,n}(n)=\frac{\varphi_1'(n)}{\varphi_1(n)}.        \label{III.cond.alpha1n}
	\end{equation}
	Looking at \eqref{III.cond.alpha1n}, our immediate choice for such an $\alpha_{1,n}$ is
	$$\alpha_{1,n}(r):=\frac{\varphi_1'(n)}{\varphi_1(n)}(r-n)+1,\quad \forall r\in\mathbb{R}.$$
	Then, by \eqref{III.est.phi1} we have for all $r\in [n,2n],$
	$$|\alpha_{1,n}(r)|\leq\left|\frac{\varphi_1'(n)}{\varphi_1(n)}\right|n+1\leq C_1+1,\ |\alpha'_{1,n}(r)|\leq \left|\frac{\varphi_1'(n)}{\varphi_1(n)}\right|\leq\frac{C_1}{n}.$$
	Clearly, $\psi_{1,n}$ belongs to $C^\infty(\mathbb{R})$ and by combining the last estimate and \eqref{III.est.phi4,n} we deduce that for all $n\in \mathbb{N}\cap (r_1,\infty),$ we have
	\begin{equation}\label{III.est.psi4,n}
	|\psi'_{1,n}(r)|\leq \frac{\gamma_0 C_1(C_1+1)}{n^{\frac{N-1}{p-1}}}+\frac{C_1^2}{n^{\frac{N-1}{p-1}}}=\frac{\widetilde{C}_1}{n^{\frac{N-1}{p-1}}},\quad \forall r\in [n,2n],
	\end{equation}
where $\widetilde{C}_1:=\gamma_0 C_1(C_1+1)+C_1^2.$ For each $n\in\mathbb{N},$ define
	\begin{equation*}\label{III.form.v_epsilon,n}
	u_{\epsilon,n}(r):=\begin{cases}
\varphi_1(r),\quad r\in [1,n],\\
	\psi_{1,n}(r),\quad r\in [n,2n],\\
	0,\quad r\in [2n,\infty),
	\end{cases}
	\end{equation*}
	and define $u_{\epsilon,n}(x):=u_{\epsilon,n}(|x|)$ for $x\in B_1^c.$
	Then by \eqref{III.cond.psi1,n}, $u_{\epsilon,n}\in C^1(B_1^c)$ and \eqref{III.est.psi4,n}, we obtain
	\begin{align*}
	\int_{B_1^c}|\nabla u_{\epsilon,n}-\nabla \varphi_1|^p\diff x&=\int_{B_n^c}|\nabla u_{\epsilon,n}-\nabla \varphi_1|^p\diff x\\
	& \leq 2^{p-1}\int_{B_n^c}|\nabla \varphi_1|^p\diff x+2^{p-1}\sigma(S_1)\int_{n}^{2n}|\psi'_{1,n}(r)|^pr^{N-1}\diff r\\
	& \leq 2^{p-1}\int_{B_n^c}|\nabla\varphi_1|^p\diff x+2^{p-1}\sigma(S_1)(2n)^{N-1}\frac{\widetilde{C}_1^p}{n^{\frac{(N-1)p}{p-1}}}n,\quad \forall n\in\mathbb{N}\cap (r_1,\infty),
	\end{align*}
where $\sigma(S_1)$ is the surface area of the unit sphere $S_1.$ Thus, $$\int_{B_1^c}|\nabla u_{\epsilon,n}-\nabla \varphi_1|^p\diff x \leq 2^{p-1}\int_{B_n^c}|\nabla\varphi_1|^p\diff x+2^{N+p-2}\sigma(S_1)\frac{\widetilde{C}_1^p}{n^{\frac{N-p}{p-1}}}\ \to\ 0\ \text{as}\ n\to\infty.$$
	Hence for some  fixed $n_1\in\mathbb{N}\cap(r_1,\infty),$ $u_\epsilon:=u_{\epsilon,n_1}$ satisfies that $u_\epsilon\in C^1(B_1^c),$ $u_\epsilon=\varphi_1$ in $A_1^{n_1}$, $u_\epsilon=0$ in $B_{2n_1}^c,$  and 
	\begin{equation}\label{III.est.v_epsilon1}
	\|u_\epsilon-\varphi_1\|<\frac{\epsilon}{3}.       
	\end{equation}
    		
	\textit{Step2: cut-off by zero near the $\partial B_1$.} Fix $\delta>0$ such that $1<r_0-2\delta<r_0+2\delta<n_1.$ Since $\varphi_1'(r)>0$ in $[1,r_0-2\delta]$ and $\varphi_1(1)=0$ we deduce
	$$C_2(r-1)\leq \varphi_1(r)\leq C_3 (r-1),\quad \forall r\in [1,r_0-2\delta],$$
	for some positive constants $C_2$ and $C_3.$ Using this estimate and noticing  $u_\epsilon=\varphi_1$ in $A_1^{n_1}$, 	
	 we can construct $v_\epsilon\in C^1(B_1^c),$ $v_\epsilon=0$ in $A_1^{1+1/n_2}$ for some $n_2\in \mathbb{N}\cap \left(\frac{2}{r_0-2\delta-1},\infty\right)$, $v_\epsilon=u_\epsilon$ in   $B_{1+2/n_{2}}^c,$  and
	\begin{equation}\label{III.est.v_epsilon}
	\|v_\epsilon-u_\epsilon\|<\frac{\epsilon}{3},
	\end{equation}
	in a similar manner to Step 1. Note that $v_\epsilon$ satisfies that $v_\epsilon\in C^1(B_1^c),$ $v_\epsilon=0$ in $A_1^{1+1/n_2}\cup B_{2n_1}^c,$ and $v_\epsilon=\varphi_1$ in $A_{r_0-2\delta}^{r_0+2\delta}.$
	
    \textit{Step 3: cut-off by a constant near $S_{r_0}$.} Analogously, for each $n\in\mathbb{N},$ we find $\psi_{2,n}, \psi_{3,n}\in C^\infty(\mathbb{R})$ such that
    \begin{equation}\label{III.cond.psi2,n}
    \begin{cases}
    \psi_{2,n}(r_0-\frac{2\delta}{n})=v_\epsilon(r_0-\frac{2\delta}{n}),\psi_{2,n}'(r_0-\frac{2\delta}{n})=v_\epsilon'(r_0-\frac{2\delta}{n}),\\
    \psi_{2,n}(r_0-\frac{\delta}{n})=v_\epsilon(r_0-\frac{\delta}{n}),\psi_{2,n}'(r_0-\frac{\delta}{n})=0,
    \end{cases}
    \end{equation}
    and 
    \begin{equation}\label{III.cond.psi3,n}
    \begin{cases}
    \psi_{3,n}(r_0+\frac{\delta}{n})=v_\epsilon(r_0-\frac{\delta}{n}),\psi_{3,n}'(r_0+\frac{\delta}{n})=0,\\
    \psi_{3,n}(r_0+\frac{2\delta}{n})=v_\epsilon(r_0+\frac{2\delta}{n}),\psi_{3,n}'(r_0+\frac{2\delta}{n})=v'_\epsilon(r_0+\frac{2\delta}{n}).
    \end{cases}
    \end{equation}
   Moreover, for all $n\in \mathbb{N}$, we have
    \begin{equation}\label{III.est.psi2,n}
    |\psi'_{2,n}(r)|\leq M, \quad \forall r\in [r_0-\frac{2\delta}{n},r_0-\frac{\delta}{n}],
        \end{equation}
        and 
        \begin{equation}\label{III.est.psi3,n}
        |\psi'_{3,n}(r)|\leq M, \quad \forall r\in [r_0+\frac{\delta}{n},r_0+\frac{2\delta}{n}],
        \end{equation}
  for some constant $M>0$ independent of $n.$
  Finally, for each $n\in\mathbb{N}$ define
  \begin{equation}\label{III.form.w_epsilon,n}
  w_{\epsilon,n}(r):=\begin{cases}
  v_\epsilon(r),\quad r\in [1,r_0-\frac{2\delta}{n}]\cup [r_0+\frac{2\delta}{n},\infty),\\
  \psi_{2,n}(r),\quad r\in [r_0-\frac{2\delta}{n},r_0-\frac{\delta}{n}],\\
  v_\epsilon(r_0-\frac{\delta}{n}),\quad r\in [r_0-\frac{\delta}{n},r_0+\frac{\delta}{n}],\\
  \psi_{3,n}(r),\quad r\in [r_0+\frac{\delta}{n},r_0+\frac{2\delta}{n}],
  \end{cases}
  \end{equation}
 and define  $w_{\epsilon,n}(x):= w_{\epsilon,n}(|x|)$ for $x\in B_1^c.$ Then by \eqref{III.cond.psi2,n} and \eqref{III.cond.psi3,n},  for all $n\in\mathbb{N}$ we have that $w_{\epsilon,n}\in C^1(B_1^c),$ $w_{\epsilon,n}=0$ in $A_1^{1+1/n_2}\cup B_{2n_1}^c,$  $w_\epsilon\equiv \operatorname{constant}$ in   $A_{r_0-\frac{\delta}{n}}^{r_0+\frac{\delta}{n}}.$  From \eqref{III.est.psi2,n}-\eqref{III.form.w_epsilon,n}, for large $n_3\in\mathbb{N}$, we have  
    	\begin{equation*}
	\|w_{\epsilon,n_3}-v_\epsilon\|<\frac{\epsilon}{3}.
	\end{equation*}
By this, \eqref{III.est.v_epsilon1}, and \eqref{III.est.v_epsilon}, we deduce that $w_\epsilon:=w_{\epsilon,n_3}$ belongs to $Y$ and satisfies \eqref{III.approximation}. The proof is complete.

\end{proof}

\end{footnotesize}

\subsection*{Acknowledgements}
The authors were supported by the project \emph{LO1506 of the Czech Ministry of Education, Youth and Sports}.

\end{document}